\documentclass[12pt]{amsart}
\usepackage[applemac]{inputenc}
\usepackage[top=3cm, bottom=2cm, left=3cm, right=3cm]{geometry}                
\usepackage{etex}
\usepackage{graphicx}
\usepackage{tabulary}
\usepackage{tabu}
\usepackage{longtable}
\usepackage{multirow}
\usepackage{booktabs}
\usepackage{amssymb}
\usepackage{tabu}
\usepackage{epstopdf}
\usepackage{appendix}
\usepackage{arydshln}
\usepackage{setspace}
\usepackage{amsmath}
\usepackage{cancel}
\usepackage{afterpage}
\usepackage{rotating}
\usepackage{pdflscape}
\usepackage{color}
\usepackage{chngcntr}
\newtheorem{theo}{Theorem}
\newtheorem*{theos}{Theorem}
\usepackage{caption}

 \theoremstyle{plain}

\newtheorem{lem}[theo]{Lemma}

\newtheorem*{cors}{Corollary}
\newtheorem{prop}[theo]{Proposition}

\def\C{\mathbb C}
\def\F{\mathbb F}

\def\Z{\mathbb Z}

\def\bG{\mathbf G}

\def\bT{\mathbf T}

\def\bL{\mathbf L}

\def\cC{\mathcal C}

\def\cC{\mathcal C}

\def\cS{\mathcal S}

\def\rA{\mathrm A}
\def\rB{\mathrm B}
\def\rC{\mathrm C}
\def\rD{\mathrm D}

\def\rF{\mathrm F}
\def\rG{\mathrm G}

\def\fC{\mathfrak C}

\def\ua{{\underline a}}
\def\ba{{\bar a}}
\def\ub{{\underline b}}

\def\uh{{\underline h}}
\def\ut{{\underline t}}
\def\ux{{\underline x}}

\def\ult{{\leadsto_u}}
\def\Wlt{{\leadsto_W}}
\def\GL{\operatorname{GL}}

\def\Ind{\operatorname{Ind}}

\def\Irr{\operatorname{Irr}}

\def\Sp{\operatorname{Sp}}
\def\SU{\operatorname{SU}}
\def\SO{\operatorname{SO}}

\def\Tr{\operatorname{Tr}}

\def\exp{\operatorname{exp}}
\counterwithout{table}{section}
\usepackage{tikz}

\def\DynkinArrowLength{3mm}
\usetikzlibrary{arrows,decorations.markings}
\usetikzlibrary{arrows,decorations.markings}
\usetikzlibrary{shapes}
\usetikzlibrary{matrix}
\usetikzlibrary{graphs}
\usetikzlibrary{backgrounds}

\title{On the decomposition numbers of $\SO_8^+(2^f)$}
\author{Alessandro Paolini}

\address{FB Mathematik, TU Kaiserslautern, Postfach 3049, 67653 Kaiserslautern, Germany.} \email{paolini@mathematik.uni-kl.de}

\thanks{Date: \today. \\
2010 \emph{Mathematics Subject Classification}. Primary 20C33, 20C15; Secondary 20G40. \\
\emph{Key words and phrases}: decomposition numbers, fusion of conjugacy classes, bad primes.\\
The author acknowledges financial support from the ERC Advanced Grant 291512 and the SFB-TRR 195.
}

\begin{document}

\maketitle

\begin{abstract} 
Let $q=2^f$, and let $G=\SO_8^+(q)$ and $U$ be 
a Sylow $2$-subgroup of $G$. We 
first describe the fusion of the conjugacy classes of $U$ in $G$. 
We then use this information to prove the unitriangularity of the 
$\ell$-decomposition matrices of $G$ for all $\ell \ne 2$ by inducing certain irreducible characters 
of $U$ to $G$; the characters of $U$ of degree $q^3/2$ play here a major role. 
We then determine the $\ell$-decomposition matrix 
of $G$ in the case $\ell \mid q+1$, when $\ell \ge 5$ and $(q+1)_\ell>5$, up to two non-negative indeterminates 
in one column. 
\end{abstract}

\section{Introduction} \label{sec:intro}

A fundamental problem in the representation theory of 
a finite group of Lie type $G$ defined over the field $\F_q$ is to determine its decomposition 
numbers in cross-characteristic $\ell$, where $\ell$ divides the order of $G$. 
These numbers relate the ordinary representations of $G$ 
with its modular representations over a field of characteristic $\ell$. 
The unipotent characters of $G$ play here a 
major role, namely if $\ell$ is large enough then 
they form a basic set for the union of 
the unipotent $\ell$-blocks \cite{GH91, Gec93}. 
A long-standing conjecture which is in general wide open \cite[(3.4)]{GH97a}
states that the restrictions of the $\ell$-decomposition matrices to the set of unipotent 
characters have unitriangular shape.

There has been significant progress towards a determination of the decomposition numbers of $G$ over the last three 
decades. 
Let $q$ be a power of a prime $p$ with $p \ne \ell$. 
The case $\ell \mid q-1$ has essentially been solved 
by means of the decomposition matrices of 
$q$-Schur algebras \cite{GH97b} and by investigating 
source algebras \cite{Pui90}; similar methods give the decomposition numbers of $\rA_n$ 
for small $n$ \cite{Jam90}. Several authors then contributed to determine 
the decomposition numbers 
of groups of small rank \cite{Him11, HH13, HN14, HN15, His89, OW98, OW02, Wak04, Whi90a, Whi90b, Whi95, Whi00}, 
for instance 
by using the Green correspondence and by investigating 
the Loewy structure of certain indecomposable modules. In the cases where $p$ is a 
bad prime for $G$, the knowledge of the character tables of parabolic subgroups is 
crucial. 

We embark in this work on the project of determining 
the unitriangularity of $\ell$-decomposition 
matrices of $G$ when $p$ is a bad prime, independently of 
the knowledge of the character tables of parabolic subgroups. 
The main result confirms the conjecture in 
\cite[(3.4)]{GH97a} in the case of $\SO_8^+(2^f)$. 

\begin{theos}
Let $\ell \ne 2$ be a prime number. Then the restrictions of the  $\ell$-decomposition matrices 
of $\SO_8^+(2^f)$ to the set of unipotent characters have unitriangular shape. 
\end{theos}

Let us now denote by $G$ the group $\SO_8^+(2^f)$, and let $U$ be a Sylow 
$2$-subgroup of $G$. The knowledge of the irreducible characters of $U$ and 
their values \cite{HLM11, LM15, GLM17} is our main ingredient to obtain the unitriangularity of the $\ell$-decomposition matrices of $G$. 
Namely we first compute the fusion of the conjugacy classes of $U$ in $G$, and 
then we use this information to construct $\ell$-projective characters by inducing 
irreducible characters of $U$ to $G$. The key fact is that the $\ell$-projective characters $\Psi_6, \dots, \Psi_9$ and $\Psi_{13}$ of $\SO_8^+(p^f)$ with $p \ge 3$ as in \cite[Section 5]{GP92}, 
constructed in terms of generalized Gelfand-Graev characters and hence defined just for $p$ odd, 
are here replaced by inducing irreducible characters of $U$ of degree $q^3/2$; this is the only irreducible 
character degree in a Sylow $p$-subgroup of $\SO_8^+(p^f)$ for any $p\ge2$ which is not a power of $p^f$. 

We then move on to computing the $\ell$-decomposition matrix of $G$ in the only remaining open case 
$\ell \mid q+1$, under the assumption $\ell \ge 5$. We 
provide exact values of its entries when $(q+1)_\ell>5$, except in one column, where such entries 
are determined as linear expressions involving two non-negative parameters $\alpha \le q/2$ 
and $\beta \le (3q-2)/2$. Most columns are obtained by Harish-Chandra induction 
from Levi subgroups. We obtain lower bounds for the $\ell$-decomposition numbers by using 
properties of irreducible $\ell$-characters in general position of some twisted torus. Upper bounds, 
which are in some cases sharp, are obtained by a character-theoretical consequence of \cite[Section 8]{BR03} 
stated in \cite[Lemma 1.1]{Dud13}, which is proved by deep cohomological methods; these methods 
have also been used to compute the decomposition numbers in types $\rB_3$ and $\rC_3$ \cite{HN14}, 
for unitary groups of low rank \cite{DM15}, and 
for exceptional groups of Lie type when $\ell \mid q^2+1$ \cite{DM16}. As these methods do not depend 
on $q$, the $\ell$-decomposition matrix of $G$ when $\ell \mid q+1$ is the same as in the case of $\SO_8^+(p^f)$ 
when $p \ge 3$. 

\begin{cors}
Let $q=2^f$, and let $\ell \ge 5$ be a prime number such that 
$\ell \mid q+1$ and $(q+1)_\ell>5$. Then the restriction of the $\ell$-decomposition 
matrix of $\SO_8^+(q)$ to the set of unipotent characters is as given in Table \ref{tab:finald}. 
\end{cors}

Notice that the matrix block of size $4$ which is indexed by the unipotent characters $\chi_i$ for $i=6, \dots, 9$, whose degrees are 
of the form $\pi_i(q)q^3/2$ for some products $\pi_i(q)$ of cyclotomic polynomials evaluated at $q$, and by 
the $\ell$-projective characters $\psi_6, \dots, \psi_9$, obtained by inducing characters of $U$ of degree $q^3/2$, is the 
only triality-stable block as in \cite[(3.4)]{GH97a} of size greater than $1$, and is in fact 
the identity block of size $4$. 
A direction for future work is to investigate in a similar way the induction of irreducible characters of degree $q^7/2$ 
in a Sylow $2$-subgroup of $\rD_5(q)$ for $q=2^f$, and of degree $q^4/3$ in a Sylow $3$-subgroup of $\rF_4(q)$ for $q=3^f$, whose construction is determined 
in \cite{Pao16} and \cite{GLMP16} respectively, to get information on the shape of the $\ell$-decomposition matrices in these cases. 
Such character 
degrees turn out to be again exactly the non-cyclotomic parts of the degrees 
of certain unipotent characters of $\rD_5(q)$ and $\rF_4(q)$ respectively, see \cite[\S\S13.8-13.9]{Car}. 

The structure of this work is as follows. We recall in Section \ref{sec:preli} some preliminary 
results on character theory and finite groups of Lie type, in particular about $\SO_8^+(2^f)$. 
In Section \ref{sec:fusio} we determine the fusion in $\SO_8^+(2^f)$ of the conjugacy classes 
of a fixed Sylow $2$-subgroup. This is used in Section \ref{sec:unitr} to find 
$\ell$-projective characters of $\SO_8^+(2^f)$ and to determine the unitriangularity of its 
$\ell$-decomposition matrices. We determine in Section \ref{sec:decnu} more precisely the 
$\ell$-decomposition matrix of $\SO_8^+(2^f)$ when $\ell \mid q+1$. Finally, we collect in the Appendix 
the details of the computations in Section \ref{sec:fusio}. Further computations related to Sections \ref{sec:unitr} and \ref{sec:decnu} are available on the webpage of the author. 

\begin{center}
\begin{table}[t]
\begin{tabular}{|c|cccccccccccccc|}
\hline
& $\Psi_1$ & $\Psi_2$ & $\Psi_3$ & $\Psi_4$ & $\Psi_5$ & $\Psi_6$ & $\Psi_7$ & $\Psi_8$& $\Psi_9$ & $\Psi_{10}$ & $\Psi_{11}$ & $\Psi_{12}$ & $\Psi_{13}$ & $\Psi_{14}$ \\
\hline
$\chi_1$
& $1$
& \multicolumn{1}{:c}{$\cdot$}
& $\cdot$
& $\cdot$
& $\cdot$
& $\cdot$
& $\cdot$
& $\cdot$
& $\cdot$
& $\cdot$
& $\cdot$
& $\cdot$
& $\cdot$
& $\cdot$
\\
\cdashline{1-3}
$\chi_2$ 
& $2$
& \multicolumn{1}{:c}{$1$}
& \multicolumn{1}{:c}{$\cdot$}
& $\cdot$
& $\cdot$
& $\cdot$
& $\cdot$
& $\cdot$
& $\cdot$
& $\cdot$
& $\cdot$
& $\cdot$
& $\cdot$
& $\cdot$
\\
\cdashline{1-1} \cdashline{3-6}
$\chi_3$ 
& $1$
& $1$
&\multicolumn{1}{:c}{$1$} 
& $\cdot$
& $\cdot$
& \multicolumn{1}{:c}{$\cdot$}
& $\cdot$
& $\cdot$
& $\cdot$
& $\cdot$
& $\cdot$
& $\cdot$
& $\cdot$
& $\cdot$
\\
$\chi_4$ 
& $1$
& $1$
&\multicolumn{1}{:c}{$0$} 
& $1$
& $\cdot$
& \multicolumn{1}{:c}{$\cdot$}
& $\cdot$
& $\cdot$
& $\cdot$
& $\cdot$
& $\cdot$
& $\cdot$
& $\cdot$
& $\cdot$
\\
$\chi_5$ 
& $1$
& $1$
&\multicolumn{1}{:c}{$0$} 
& $0$
& $1$
& \multicolumn{1}{:c}{$\cdot$}
& $\cdot$
& $\cdot$
& $\cdot$
& $\cdot$
& $\cdot$
& $\cdot$
& $\cdot$
& $\cdot$
\\
\cdashline{1-1} \cdashline{4-10}
$\chi_6$ 
& $0$
& $1$
& $1$
& $1$
& $1$
& \multicolumn{1}{:c}{$1$}
& $\cdot$
& $\cdot$
& $\cdot$
& \multicolumn{1}{:c}{$\cdot$}
& $\cdot$
& $\cdot$
& $\cdot$
& $\cdot$
\\
$\chi_7$ 
& $2$
& $2$
& $1$
& $1$
& $1$
& \multicolumn{1}{:c}{$0$}
& $1$
& $\cdot$
& $\cdot$
& \multicolumn{1}{:c}{$\cdot$}
& $\cdot$
& $\cdot$
& $\cdot$
& $\cdot$
\\
$\chi_8$ 
& $ 0 $
& $0$
& $0$
& $0$
& $0$
& \multicolumn{1}{:c}{$0$}
& $0$
& $1$
& $\cdot$
& \multicolumn{1}{:c}{$\cdot$}
& $\cdot$
& $\cdot$
& $\cdot$
& $\cdot$
\\
$\chi_9$ 
& $0$
& $0$
& $0$
& $0$
& $0$
& \multicolumn{1}{:c}{$0$} 
& $0$
& $0$
& $1$
& \multicolumn{1}{:c}{$\cdot$}
& $\cdot$
& $\cdot$
& $\cdot$
& $\cdot$
\\
\cdashline{1-1} \cdashline{7-13}
$\chi_{10}$ 
& $1$
& $1$
& $1$
& $1$
& $1$
& $1$ 
& $1$
& $0$
& $\alpha$
& \multicolumn{1}{:c}{$1$}
& $\cdot$
& $\cdot$
& \multicolumn{1}{:c}{$\cdot$}
& $\cdot$
\\
$\chi_{11}$ 
& $1$
& $1$
& $1$
& $1$
& $1$
& $1$ 
& $1$
& $0$
& $\alpha$
& \multicolumn{1}{:c}{$0$}
& $1$
& $\cdot$
& \multicolumn{1}{:c}{$\cdot$}
& $\cdot$
\\
$\chi_{12}$ 
& $1$
& $1$
& $1$ 
& $1$
& $1$
& $1$ 
& $1$
& $0$
& $\alpha$
& \multicolumn{1}{:c}{$0$}
& $0$
& $1$
& \multicolumn{1}{:c}{$\cdot$}
& $\cdot$
\\
\cdashline{1-1} \cdashline{11-14}
$\chi_{13}$ 
& $2$
& $1$
& $1$
& $1$
& $1$
& $1$ 
& $2$ 
& $0$
& $\beta$
& $1$
& $1$
& $1$
& \multicolumn{1}{:c}{$1$}
& \multicolumn{1}{:c|}{$\cdot$}
\\
\cdashline{1-1} \cdashline{14-15}
$\chi_{14}$
& $1$
& $0$
& $0$
& $0$
& $0$
& $1$ 
& $1$
& $0$
& $\gamma$
& $1$
& $1$
& $1$
& $4$
& \multicolumn{1}{:c|}{$1$}
\\
\hline
\end{tabular}
\caption{The $\ell$-decomposition matrix of $G$ for $\ell \mid q+1$, when $\ell \ge 5$ and $(q+1)_\ell>5$. 
We have that $\gamma=-9\alpha+4\beta+8$, and that $\alpha \le q/2$, $\beta 
\le (3q-2)/2$ and $-9\alpha+4\beta+8 \ge 0$.}
\label{tab:finald}
\end{table}
\end{center}

\noindent
\textbf{Acknowledgement:} The author deeply thanks G. Malle for very helpful 
feedback and comments on an earlier version of the work, and O. Dudas for 
further comments and suggestions. 

\section{Preliminaries} \label{sec:preli}

We recall first some notation for characters and 
conjugacy classes of a finite group $G$ and its 
subgroups. Our main reference for this is \cite{Is}. We denote by $\Irr(G)$ the set of 
irreducible characters of $G$. Let $H$ be a subgroup of $G$. For 
a character $\chi$ of $H$, we denote by $\Ind_H^G \chi$, 
or more simply $\chi^G$, the character obtained by 
inducing $\chi$ to $G$. If $g, h \in G$, we denote by 
$h^g$ the element $g^{-1}hg$. In this way, the set 
$h^G$ denotes the set of all elements of the form 
$h^g$ when $g \in G$, that is, the conjugacy class 
of $G$ containing $h$. 

Let now $\varphi \in \Irr(H)$, and let $\{h_1, 
\dots, h_m\}$ be a full set of representatives of the set $\{h^g \mid h \in H \text{ and } g \in G\}$ for the 
conjugation action in $G$. The information of 
the fusion of the conjugacy classes of $H$ into $G$ 
is enough to compute the character values of $\varphi^G$ on each $h_k^g$ with $k=1, \dots, m$ and $g \in G$. 
More in details, if 
$g \in h_k^G$ and 
$$h_k^G \cap H = (h_k^1)^H \sqcup \dots \sqcup (h_k^{\ell(k)})^H,$$
with $h_k^1, \dots, h_k^{\ell(k)}$ a complete 
set of representatives of $H$-conjugacy classes 
in $h_k^G \cap H$, then as remarked in \cite[Section 5]{Is} we have 
\begin{equation}\label{eq:isa}
\varphi^G(g)= |C_G(g)| \sum_{i=1}^{\ell(k)} \frac{\varphi(h_k^{i})}{|C_H(h_k^{i})|}.
\end{equation}

We now recall the definition of the group $\SO_8(2^f)$ and some of its properties, following \cite{MT}. 
Let $k$ be an algebraically closed field of characteristic 
$2$, and for a fixed $f$ let $q=2^f$ and let $F$ be the 
standard Frobenius morphism on $k$. Hence the set of 
fixed points of $F$ on $k$ is the field $\F_q$ with $q$ elements. 
Let $\bG$ be a 
simple algebraic group of type $\rD_4$ defined over $k$. 
The Frobenius morphism $F$ also acts on $\bG$ through a 
standard linear embedding. From now on, $G$ is defined to be the group $\SO_8^+(q)$ of fixed points of $\bG$ under $F$. The group $G$ is then a finite 
Chevalley group of order 
$$|G|=q^{12}\phi_1^4\phi_2^4\phi_3\phi_4^2\phi_6.$$
Here we put $\phi_i:=\phi_i(q)$, where 
$\phi_i(x)$ denotes the $i$-th cyclotomic 
polynomial. 

Let us fix a maximally split torus $\bT$ of $\bG$, and let $T=\bT^F$. We denote by $\Phi$ the 
root system of $\bG$ with respect to $\bT$, and by $\Phi^+=\{\alpha_1, \dots, \alpha_{12}\}$ 
the set of positive roots, with $\alpha_1, \dots, \alpha_4$ simple roots. 
The set $\Phi$ is acted on by the Weyl group $W=N_\bG(\bT)/\bT$, which coincides in this case with 
$N_G(T)/T$. We denote by $s_1, \dots, s_4$ the set of 
standard generators for $W$ such that $s_i(\alpha_i)=-\alpha_i$ for $i=1, \dots, 4$. 
Let $U$ be a fixed Sylow $2$-subgroup of $G$. We denote by $X_\alpha$ the root subgroup 
of $U$ with respect to a root $\alpha \in \Phi^+$, and by $x_\alpha(t)$ the root element 
corresponding to $\alpha \in \Phi^+$ and $t \in \F_q$. If $i \in \{1, \dots, 12\}$, we often 
write $X_i$ instead of $X_{\alpha_i}$, and $x_i(t)$ in place of $x_{\alpha_i}(t)$ for 
$t \in \F_q$. 

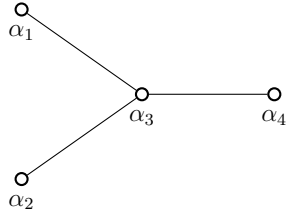
\begin{figure}[h]
\begin{center}
\begin{tikzpicture}[transform shape, scale=0.8, place/.style={circle,draw=black,fill=black, tiny},middlearrow/.style={
    decoration={markings,
      mark=at position 0.6 with
      {\draw (0:0mm) -- +(+135:\DynkinArrowLength); \draw (0:0mm) -- +(-135:\DynkinArrowLength);},
    },
    postaction={decorate}
  }, dedge/.style={
    middlearrow,
    double distance=0.5mm,
  }]
  \node (a) at (-2,1.41) [circle, draw, thick, fill=none, inner sep=2pt,label=below:$\alpha_1$] {};
  \node (b) at (-2,-1.41) [circle, draw, thick, fill=none, inner sep=2pt,label=below:$\alpha_2$] {};
    \node (c) at (0,0) [circle, draw, thick, fill=none, inner sep=2pt,label=below:$\alpha_3$] {};
  \node(d) at (2.2,0) [circle, draw, thick, fill=none, inner sep=2pt,label=below:$\alpha_4$] {};
    \draw (a) -- (c);
  \draw (b) -- (c) ;
  \draw (c) -- (d) ;
\end{tikzpicture}
\caption{The Dynkin diagram of type $\rD_4$. Simple roots are labelled as in CHEVIE.}
\label{tab:E6}
\end{center}
\end{figure}

We recall some of the Chevalley relations in $G$, as in \cite[Theorem 1.12.1]{GLS3}. 
Every element of $T$ can be described as $\uh(\ut):=h_1(t_1)\cdots h_4(t_4)$ with $t_1, \dots, t_4 \in \F_q^\times$, 
for certain $h_i: \F_q^\times \to T$ for $i=1, \dots, 4$. The action of $\uh(\ut)$ on
an element $\ux(\ua):=x_3(a_3) x_1(a_1) \cdots x_{12}(a_{12})$ of $U$, 
for $t_1, \dots, t_4, a_1, \dots, a_{12} \in \F_q^{\times}$, is given by 
\begin{align}\label{eq:al3}
\ux (\ua) ^{\uh(\ut)}=&x_3(a_3t_1t_2t_4/t_3^2)x_1(a_1t_3/t_1^2)
 x_2(a_2t_3/t_2^2)x_4(a_4t_3/t_4^2)
x_5(a_5t_2t_4/t_1t_3)\cdot \\
& 
x_6(a_6t_1t_4/t_2t_3)x_7(a_7t_1t_2/t_3t_4)x_8(a_8t_4/t_1t_2)x_9(a_9t_2/t_1t_4)\cdot \nonumber \\ 
&x_{10}(a_{10}t_1/t_2t_4) x_{11}(a_{11}t_3/t_1t_2t_4)
x_{12}(a_{12}/t_3). \nonumber
\end{align}

Since $q$ is a $2$-power, we have that 
\begin{equation}\label{eq:Wact}
x_{\alpha}(t)^{\dot{w}}=x_{w(\alpha)}(t)
\end{equation}
for $\alpha \in \Phi^+$, $\dot{w}$ a lift in $N_G(T)$ of $w \in W$ and $t \in \F_q$.  
Moreover, we use the general formula 
$$
[x_{\alpha}(s), x_{\beta}(t)]=\prod_{\substack{i, j \in \mathbb{Z}_{>0} \mid  i\alpha+j\beta \in \Phi^+}}x_{i\alpha+j\beta}(
c_{i,j}^{\alpha,\beta}(-t)^is^j)
$$
for $\alpha, \beta \in \Phi^+$ and $s, t \in \F_q$, to obtain  
\begin{equation}\label{eq:Uact}
[x_{\alpha}(s), x_{\beta}(t)]=
\begin{cases}
x_{\alpha+\beta}(st) & \text{ if }\alpha+\beta \in \Phi^+, \\
0 & \text{ otherwise. }
\end{cases}
\end{equation}

We denote by $\Tr:\F_q \to \F_2$ the trace map of the field $\F_q$ 
over its prime field $\F_2$. We fix a non-trivial character of the 
abelian group $\F_q$, namely 
$\phi:\F_q \to \C^\times$ such that $\phi(x)=\exp(i\pi\Tr(x))$. We have that 
$\ker(\phi)=\{t^2+t \mid t \in \F_q\}$. 

There are $2q^5+8q^4-16q^3+14q^2-10q+3$ distinct conjugacy 
classes in $U$. These have been obtained in \cite[Section 3]{BG14}; 
we consider in Section \ref{sec:fusio} the same conjugacy class representatives as in \cite[Section 3]{GLM17}. 
On the other hand, there are just $14$ unipotent class 
representatives in $G$, that is, representatives of the $G$-conjugacy 
action on the set $\{u^g \mid u \in U \text{ and } g \in G\}$. These are 
obtained in \cite{GLO17}, and can also be recovered from 
the computations in Section \ref{sec:fusio}. We fix  
$u_1, \dots, u_{14}$ such representatives in $U$. These are collected 
in Table \ref{tab:repU}, together with their centralizer sizes in $G$. 
In particular, we fix there from now on the element 
$\mu \in \F_q^\times$ 
such that $\mu$ is an element with $\Tr(\mu)=1$ if $q>2$, and $\mu=1$ if $q=2$. 
Notice that $u_3$, $u_4$ and $u_5$ 
(respectively $u_9$, $u_{10}$ and $u_{11}$) are 
permuted by the triality automorphism 
$\tau$ of $G$, which is associated to the Dynkin diagram automorphism $(\alpha_1, \alpha_4, \alpha_2)$. As a consequence of 
subsequent calculations, we can determine precisely that 
the classes $u_{13}$ and $u_{14}$ correspond to the 
classes denoted respectively by $u_{13}$ and $u_{13}'$ in \cite[\S 4.5.4]{Gec94}; in fact, by 
swapping the roles of $u_{13}$ and $u_{14}$ we would get negative entries in Table \ref{tab:decn}, which 
of course cannot happen. 

We end this section by recalling some information about 
characters of $U$ and $G$. The set $\Irr(U)$ is described 
in \cite{HLM11} and \cite{GLM17}. Notice that the parametrization 
of the irreducible characters of a Sylow $p$-subgroup of $\SO_8^+(p^f)$ is uniform 
for primes $p \ge 3$, while 
in the case $p=2$ there is some further complication. We denote by $\chi_1, \dots, 
\chi_{14}$ the unipotent characters of $G$. We use the same 
notation as \cite{GP92}, except 
we swap $\chi_7$ and $\chi_8$. By \cite[\S 4.5.4]{Gec94}, the unipotent 
characters of $G$ and their values on $u_1, \dots, u_{12}$ are given in both cases $p\ge 3$ and $p=2$ 
by identical polynomial expressions in 
$p^f$; they are 
determined there, as well as the values on $u_{13}$ and 
$u_{14}$ when $p=2$. 

\begin{center}
\begin{table}[t]
\begin{tabular}{|c|c|c|}
\hline
Class & Representative & Centralizer order \\
\hline
\hline
$u_1$ & $1_G$ & $q^{12}(q^2-1)(q^4-1)^2(q^6-1)$ \\
\hline
$u_2$ & $x_1(1)$ & $q^{12}(q^2-1)^3 $ \\
\hline
$u_3$ & $x_1(1)x_2(1)$ & $q^{10}(q^2-1)(q^4-1)$ \\ 
\hline
$u_4$ & $x_1(1)x_4(1)$ & $q^{10}(q^2-1)(q^4-1)$ \\ 
\hline
$u_5$ & $x_2(1)x_4(1)$ & $q^{10}(q^2-1)(q^4-1)$ \\ 
\hline
$u_6$ & $x_1(1)x_2(1)x_4(1)$ & $q^{10}(q^2-1)$ \\ 
\hline
$u_7$ & $x_3(1)x_1(1)$ & $2q^8(q-1)^2$ \\ 
\hline
$u_8$ & $x_1(1)x_2(1)x_4(1)x_{10}(1)x_{12}(\mu)$ & $2q^8(q+1)^2$ \\ 
\hline
$u_9$ & $x_3(1)x_1(1)x_2(1)$ & $q^6(q^2-1)$ \\ 
\hline
$u_{10}$ & $x_3(1)x_1(1)x_4(1)$ & $q^6(q^2-1)$ \\ 
\hline
$u_{11}$ & $x_3(1)x_2(1)x_4(1)$ & $q^6(q^2-1)$ \\ 
\hline
$u_{12}$ & $x_3(1)x_1(1)x_2(1)x_{10}(1)$ & $q^6$ \\
\hline
$u_{13}$ & $x_3(1)x_1(1)x_2(1)x_4(1)$ & $2q^4$ \\ 
\hline
$u_{14}$ & $x_3(1)x_1(1)x_2(1)x_4(1)x_{10}(\mu)$ & $2q^4$ \\
\hline
\end{tabular}
\caption{The unipotent class representatives in $\SO_8^+(2^f)$
.} \label{tab:repU}
\end{table}
\end{center}
\vspace{-8mm}

\section{The fusion of the conjugacy classes of $U$ into $G$} \label{sec:fusio}

We now outline a method to determine how the 
conjugacy class representatives of $U$ in \cite[Section 3]{GLM17} fuse into the unipotent classes of $G$. 
We collect such representatives and their fusion into each of the $u_k^G$ for $k=1, \dots, 14$ 
in Table \ref{tab:long}. We make substantial use of the Chevalley relations in $G$ and properties of $\Phi^+$. 
From now on and for the rest of this work, we assume that $q>2$. Namely if $q=2$ then $\ell\in \{3, 5, 7\}$, and 
in these cases the $\ell$-Brauer characters of $\SO_8^+(2)$ are determined in \cite{JLPW}. Notice that for $q=2$ the fusion of 
the conjugacy classes of $U$ in $G$ is slightly different, as $\Tr(0)=\Tr(1)=0$; this is readily described by CHEVIE \cite{GHL+96}, 
and can also be computed by using the methods outlined in this section.

We first describe the action of $T$ on the representatives of the conjugacy classes 
of $U$ in the first row corresponding to $u_k^G$ in Table \ref{tab:long}. 
Namely by using Equation \eqref{eq:al3}, 
for $k=1, \dots, 14$ and each $u \in u_k^G$ in the first row of Table \ref{tab:long} of the form 
$x_{\alpha_{i_1}}(a_{i_1})\cdots x_{\alpha_{i_m}}(a_{i_m})$, 
we can find explicitly  $t_1, \dots, t_4 \in \F_q^\times$ such that 
$$\left( x_{\alpha_{i_1}}(a_{i_1})\cdots x_{\alpha_{i_m}}  (a_{i_m})\right)^{h_1(t_1)h_2(t_2)h_3(t_3)h_4(t_4)}=u_k.$$
This information is collected in Table \ref{tab:reptorus}, with the following notation. 
For fixed $a_i \in \F_q^{\times}$, we denote by $\omega_i=\omega_i(a_i) \in 
\F_q^{\times}$ the unique square root of $a_i$ in $\F_q^{\times}$.  
We omit the trivial case $k=1$ and the cases $k=4$ and $k=5$ (respectively 
$k=10$ and $k=11$), which are obtained by case $k=3$ (respectively 
$k=9$) by applying the triality automorphism $\tau$. 

\begin{center}
\begin{table}[t]
\scalebox{0.95}{
\begin{tabular}{|c|c|c|}
 \hline
$k$ & First row rep. in Table \ref{tab:long} for $u_k^G$   & $\ut \in (\F_q^\times)^4 \mid ( x_{\alpha_{i_1}}(t_{i_1})\cdots x_{\alpha_{i_m}}  (t_{i_m}))^{\uh(\ut)}=u_k$ \\
\hline
\hline
$2$  & $x_1(a_1)$ &  $t_1=\omega_1$, $t_2=t_3=t_4=1$ \\
\hline
$3$  & $x_1(a_1)x_4(a_4) $ &  $t_1=\omega_1$, $t_2=\omega_2$, $t_3=t_4=1$ \\
\hline
$6$  & $x_1(a_1)x_2(a_2)x_4(a_4) $ &  $t_1=\omega_1$, $t_2=\omega_2$, $t_3=1$, $t_4=\omega_4$ \\
\hline
$7$  & $x_3(a_3)x_1(a_1) $ &  $t_1=\omega_1$, $t_2=1/a_3$, $t_3=1$, $t_4=1/\omega_1$ \\
\hline
 \multirow{2}{*}{$8$}  &  $x_1(a_1)x_2(a_2)x_4(a_4)x_{10}(a_{10})\cdot  $ &  $t_1=a_1a_{10}/(\omega_2\omega_4)$, $t_2=a_{10}\omega_1/\omega_4,$ \\
 & $x_{12}(a_1a_{10}^2\mu/(a_2a_4))$ & $t_3=a_1a_{10}^2/(a_2a_4)$, $t_4=a_{10}\omega_1/\omega_2$ \\
\hline
$9$  & $x_3(a_3)x_1(a_1)x_2(a_2) $ &  $t_1=\omega_1$, $t_2=\omega_2$, $t_3=1$, $t_4=1/(\omega_1 \omega_2 a_3)$ \\ 
\hline
 \multirow{2}{*}{$12$}  &  
 \multirow{2}{*}{$x_3(a_3)x_1(a_1)x_2(a_2)x_{10}(a_{10})$} &  $t_1=a_1\omega_3\omega_{10}$, $t_2=\omega_1\omega_2\omega_3\omega_{10},$ \\
 & & $t_3=a_1a_3a_{10}$, $t_4=a_{10}\omega_1/\omega_2$ \\ 
 \hline
$13$  &  
$x_3(a_3)x_1(a_1)x_2(a_2)x_4(a_4)$  &  $t_1=a_1\omega_2 a_3\omega_4$, $t_2=\omega_1 a_2 a_3\omega_4,$ \\
 \cline{1-2}
$14$  &  $x_3(a_3)x_1(a_1)x_2(a_2)x_4(a_4)x_{10}(a_2a_3a_4\mu)$ &   
 $t_3=a_1 a_2 a_3^2 a_4$, $t_4=\omega_1\omega_2 a_3 a_4$ \\ 
\hline
\end{tabular}
}
\caption{The normalizing action of $T$ on the 
representatives in the first row corresponding to $u_k^G$, $k=1, \dots, 14$ in Table \ref{tab:long}.} \label{tab:reptorus}
\end{table}
\end{center}
\vspace{-7mm}

We then give full information on the fusion of the conjugacy classes of $U$ into $G$ 
by using the actions of $W$ and $U$ on the remaining $U$-conjugacy representatives. 
The action of $W$ is 
used to show that $x_{\beta_1}(a_1) \cdots x_{\beta_m}(a_m)$ and 
$x_{\gamma_1}(a_1) \cdots x_{\gamma_m}(a_m)$ are conjugate when 
$w(\beta_1, \dots, \beta_m)=(\gamma_1, \dots, \gamma_m)$, by using 
Equation \eqref{eq:Wact}. Here we consider the action of $W$ on 
$(\Phi^+)^m$ entrywise for some $m \ge 1$; we will see that if $\beta_i+\beta_j \notin \Phi^+$ 
for $1 \le i < j \le m$, or if certain other conditions hold, then we can replace the $m$-tuple 
$(\beta_1, \dots, \beta_m)$ with the set $\{\beta_1, \dots, \beta_m\}$, in other words the order of the roots will not matter. This motivates 
the subsequent notation of $[\beta_1, \dots, \beta_m]$. In some cases, we need to combine 
both actions of $W$ and $U$. 

We now go in more details to determine the fusion 
of the conjugacy class representatives of $U$ in $G$ as in Table \ref{tab:fusiond4}. 
Our strategy is to fix each $G$-conjugacy class representative labelling the first row corresponding to $u_k^G$ in 
Table \ref{tab:fusiond4} 
for $k=1, \dots, 14$, and to prove that every class representative 
in $U$ is conjugate to one of these $G$-conjugacy class representatives. 

We introduce some 
notation. We denote by $[(\alpha_{i_1}, a_{i_1}), \dots, (\alpha_{i_m}, a_{i_m})]$, 
or in short $[\alpha_{i_1}, \dots, \alpha_{i_m}]$, a set of roots $\alpha_{i_j} \in \Phi^+$ 
corresponding to the generic elements $a_{i_j} \in \F_q^\times$ for $1 \le j \le m$, defined up to the following 
equivalence. 

\begin{itemize}

\item[(i)] Let $\sigma \in \text{Sym}\{i_1, \dots, i_m\}$. 
We identify $[\alpha_{i_1}, \dots, \alpha_{i_m}]$ with $[\alpha_{\sigma(i_1)}, \cdots, \alpha_{\sigma(i_m)}]$, 
that is, we put $[\alpha_{i_1}, \dots, \alpha_{i_m}]=[\alpha_{\sigma(i_1)}, \cdots, \alpha_{\sigma(i_m)}]$, if 
$$x_{i_1}(a_{i_1}) \cdots x_{i_m}(a_{i_m})=x_{\sigma(i_1)}(a_{\sigma(i_1)}) \cdots x_{\sigma(i_m)}(a_{\sigma(i_m)})$$
for every $a_{i_1}, \dots, a_{i_m} \in \F_q^\times$. 

\item[(ii)] We write $[(\alpha_{i_1}, a_{i_1}), \dots, (\alpha_{i_m}, a_{i_m})] \ult [(\alpha_{j_1}, a_{j_1}'), \dots, (\alpha_{j_\ell}, a_{j_\ell}')]$ 
for some $u \in U$, in short $[\alpha_{i_1}, \dots, \alpha_{i_m}]\ult [\alpha_{j_1}, \dots, \alpha_{j_\ell}]$, if 
there exist $a_{j_1}', \dots, a_{j_\ell}' \in 
\F_q^{\times}$ such that 
$$\left( x_{i_1}(a_{i_1}) \cdots x_{i_m}(a_{i_m}) \right)^u=
x_{j_1}(a_{j_1}') \cdots x_{j_\ell}(a_{j_\ell}').$$

\item[(iii)] We write $[\alpha_{i_1}, \dots, \alpha_{i_m}] \leadsto_W [\alpha_{j_1}, \dots, \alpha_{j_m}]$ 
if 
there exists $w \in W$ such that 
$$w(\alpha_{i_1})=\alpha_{j_1}, \qquad \dots, \qquad w(\alpha_{i_m})=\alpha_{j_m}.$$

\end{itemize}

We remark that if (iii) occurs for $w \in W$, then for every $a_{i_1}, \dots, a_{i_m} \in \F_q^\times$ there exists a lift 
$\dot{w} \in N_G(T)$ of $w$ such that 
$$\left( x_{i_1}(a_{i_1}) \cdots x_{i_m}(a_{i_m}) \right)^{\dot{w}}=
x_{i_1}(a_{i_1})^{\dot{w}}  \cdots x_{i_m}(a_{i_m})^{\dot{w}}=
x_{j_1}(a_{i_1}) \cdots x_{j_m}(a_{i_m}).$$
Finally, we observe that an element of the form 
$ x_{i_1}(a_{i_1}) \cdots x_{i_m}(a_{i_m})$ is conjugate 
in $G$ to some element of the form 
$x_{j_1}(a_{j_1}') \cdots x_{j_\ell}(a_{j_\ell}')$ if 
we can transform $[\alpha_{i_1}, \dots, \alpha_{i_m}]$ into 
$[\alpha_{j_1}, \dots, \alpha_{j_\ell}]$ by using (i), (ii) or (iii) above. 
In this case, 
we will write
$$[\alpha_{i_1}, \dots, \alpha_{i_m}] \leadsto [\alpha_{j_1}, \dots, \alpha_{j_\ell}].$$

The rest of this section is devoted to obtain $[\alpha_{i_1}, \dots, \alpha_{i_m}] 
\leadsto [\alpha_{j_1}, \dots, \alpha_{j_\ell}]$ for each $\alpha_{i_1}, \dots, \alpha_{i_m}$ 
which index a row in Table \ref{tab:fusiond4} corresponding to some $u_k^G$, for $k \in \{1, \dots, 14\}$, and 
$\alpha_{j_1}, \dots, \alpha_{j_\ell}$ indexing the first row associated to $u_k^G$ in Table \ref{tab:fusiond4}. 
This gives already almost all the fusion of the $U$-conjugacy 
representatives into $G$. Just a few representatives are then left to examine; we 
explicitly conjugate them to suitable representatives in the first row corresponding to some $u_k^G$ 
in Table \ref{tab:fusiond4}. 

We make use of the following result, of independent interest. 
We denote by $r$ a power of an arbitrary prime $p$, and by $U(\rA_n(r))$ a Sylow 
$p$-subgroup of $\rA_n(r)$.

\begin{prop}\label{prop:typeA} Let $\Phi$ be a root system of type $\rA_n$, 
with $\alpha_1, \dots, \alpha_n$ simple roots. 
If $\{\alpha_{i_1}, \dots, \alpha_{i_n}\}=\{\alpha_{1}, \dots, \alpha_n\}$, 
then for every $a_{i_1}, \dots, a_{i_n} \in \F_r^\times$ there exist $u \in U(\rA_n(r))$ and $a_{1}, \dots, a_{n} \in \F_r^\times$, 
such that 
$$\left(x_{i_1}(a_{i_1}) \dots x_{i_n}(a_{i_n})\right)^u=x_{1}(a_1) \dots x_{n}(a_{n}).$$
\end{prop}

\begin{proof}
We can rewrite the product $x_{i_1}(a_{i_1}) \dots x_{i_n}(a_{i_n})$ as 
$$x_{i_1}(a_{i_1}) \dots x_{i_n}(a_{i_n})=x_{1}(a_1) \dots x_{n}(a_{n}) \prod_{j = n+1}^m x_j(b_j)$$
for some $a_1, \dots, a_n \in \F_r^\times$ and some $b_{n+1}, \dots, b_m \in \F_r$, 
where $m=n(n+1)/2$ is the number of positive roots in type $\rA_n$. To prove the claim, it is enough to show that for every 
$k \ge n+1$ and $a_1, \dots, a_n, a_k \in \F_r^\times$, $b_{k+1}, \dots, b_m \in \F_r$, there exists 
$u \in U(\rA_n(r))$ such that 
\begin{equation}\label{eq:conj}
\left(x_{1}(a_1) \dots x_{n}(a_{n}) x_k(a_k)\prod_{j = k+1}^mx_j(b_j)\right)^u=x_{1}(a_1) \dots x_{n}(a_{n})\prod_{j = k+1}^m x_j(b_j')
\end{equation}
for some $b_{k+1}', \dots, b_m' \in \F_r$. 

Let us fix $k \ge n+1$. Then we have that $\alpha_k=\alpha_i+\alpha_{i+1}+\dots + \alpha_{\ell}$ for suitable $1 \le i < \ell \le n$. 
Let us put $\alpha_s:=\alpha_k-\alpha_i=\alpha_{i+1}+\dots + \alpha_{\ell}$. If $1 \le j \le n$, notice that $\alpha_j + \alpha_s \notin \Phi^+$ if 
$j<i$, and $\alpha_j+\alpha_s \in \{\alpha_{k+1}, \dots, \alpha_m\}$ or $\alpha_j + \alpha_s \notin \Phi^+$ if $j>i$. 
Then 
the element $u=x_s(\epsilon a_k/a_i)$ for some $\epsilon \in \{\pm 1\}$ satisfies Equation \eqref{eq:conj}. The claim follows. \end{proof}

We now start the analysis for each $k=1, \dots, 14$. We use the following 
GAP code to determine the orbit of a tuple $v$ of  
positive roots under the action of $W$. 
We realize $W$ as a suitable permutation group, which acts on 
the root labels $1, \dots, 24$. 

WeylOrbList:=function(W, v)

return Filtered(Orbit(W, v, OnTuples), x $\to$ Maximum(x) $\le$ 12);

end;

\subsection{The cases $k=1, \dots, 5$} 

As the action of $W$ is transitive on $\Phi$, we have 
$[\alpha_{i}] \Wlt [\alpha_{j}]$ for every $i, j \in \{1, \dots, 12\}$. 
For $k=3$, 
we have that 
$$[\alpha_{1}, \alpha_{2}] \Wlt [\alpha_{4}, \alpha_{12}] \Wlt [\alpha_{3}, \alpha_{8}] \Wlt [\alpha_{5}, \alpha_{6}] \Wlt 
[\alpha_{7}, \alpha_{11}] \Wlt [\alpha_{9}, \alpha_{10}].$$
The cases $k=4$ and $k=5$ are obtained by applying the triality 
automorphism $\tau$ to the case $k=3$. 

\subsection{The case $k=6$}  We have that 
\begin{align*}
 [ \alpha_{1}, \alpha_{2}, \alpha_{4} ]&\Wlt [ \alpha_{1}, \alpha_{2}, \alpha_{12} ]\Wlt [ \alpha_{1}, \alpha_{4}, \alpha_{12} ]
 \Wlt [ \alpha_{2}, \alpha_{4}, \alpha_{12} ]\Wlt [ \alpha_{5},\alpha_{ 6},\alpha_{ 7} ]
 \Wlt  \\ 
 & \Wlt [ \alpha_{5},\alpha_{ 6}, \alpha_{11} ]\Wlt 
   [ \alpha_{5}, \alpha_{7}, \alpha_{11} ]\Wlt [ \alpha_{6}, \alpha_{7}, \alpha_{11} ]\Wlt [ \alpha_{8},\alpha_{ 9}, \alpha_{10} ]
  \Wlt  \\ 
  &\Wlt [ \alpha_{3}, \alpha_{8}, \alpha_{9} ] \Wlt[ \alpha_{3}, \alpha_{8}, \alpha_{10} ]\Wlt [ \alpha_{3}, \alpha_{9}, \alpha_{10} ].
\end{align*}
Moreover, we notice that 
$[\alpha_{3}, \alpha_{8}, \alpha_{9}, \alpha_{10}] \Wlt [\alpha_{1}, \alpha_{2}, \alpha_{4}, \alpha_{12}]$, and that 
$$\left(  x_1(a_1)x_2(a_2)x_4(a_4)x_{12}(a_{12}) \right)
^{u
}
=x_1(a_1)x_2(a_2)x_4(a_4)
$$
with $u=x_5(\omega)x_6(a_2\omega/a_1)x_7(a_4\omega/a_1)
x_{10}(a_2a_4\omega/a_1)$, 
where $\omega \in \F_q^\times$ is the unique square root of 
$a_1a_{12}/(a_2a_4)$. Hence $[\alpha_{3}, \alpha_{8}, \alpha_{9}, \alpha_{10}] \Wlt 
[\alpha_{1}, \alpha_{2}, \alpha_{4}, \alpha_{12}] \ult [\alpha_{1}, \alpha_{2}, \alpha_{4}]$.

\subsection{The case $k=7$}\label{sub:k7} We have that 
\begin{align*}
 [ \alpha_{3}, \alpha_{1} ]&\Wlt [ \alpha_{3}, \alpha_{2} ]\Wlt [ \alpha_{3}, \alpha_{4} ]\Wlt [ \alpha_{1}, \alpha_{6} ]
 \Wlt [ \alpha_{1}, \alpha_{7} ]
 \Wlt 
  [ \alpha_{1}, \alpha_{10} ]\Wlt \\ 
  &\Wlt [ \alpha_{2}, \alpha_{5} ]\Wlt [ \alpha_{2}, \alpha_{7} ]
  \Wlt [ \alpha_{2},\alpha_{ 9} ]\Wlt [ \alpha_{4}, \alpha_{5} ]\Wlt [ \alpha_{4}, \alpha_{6} ]\Wlt \\ 
  &\Wlt 
  [ \alpha_{4}, \alpha_{8} ]\Wlt 
  [ \alpha_{3}, \alpha_{11} ]\Wlt [ \alpha_{5}, \alpha_{10} ] \Wlt [ \alpha_{6}, \alpha_{9} ] \Wlt [ \alpha_{7}, \alpha_{8} ],
\end{align*}
and that 
\begin{align*}
 &[ \alpha_{1}, \alpha_{2}, \alpha_{6} ]\Wlt [ \alpha_{1},\alpha_{2}, \alpha_{10} ]\Wlt [ \alpha_{3}, \alpha_{8}, \alpha_{11} ]\Wlt [ \alpha_{5}, 
 \alpha_{6}, \alpha_{10 }], \\ 
 &[ \alpha_{1}, \alpha_{4}, \alpha_{7} ]\Wlt [ \alpha_{1}, \alpha_{4}, \alpha_{10} ]\Wlt [ \alpha_{3}, \alpha_{9}, \alpha_{11} ]
 \Wlt [ \alpha_{5}, \alpha_{7}, \alpha_{10} ], \\ 
  &[ \alpha_{2}, \alpha_{4}, \alpha_{7} ]\Wlt [ \alpha_{2}, \alpha_{4}, \alpha_{9} ]\Wlt [ \alpha_{3}, \alpha_{10}, \alpha_{11} ]\Wlt 
  [ \alpha_{6}, \alpha_{7}, \alpha_{9} ], \\ 
 &[ \alpha_{1}, \alpha_{2}, \alpha_{4}, \alpha_{10} ]\Wlt [ \alpha_{5}, \alpha_{6}, \alpha_{7}, \alpha_{10 }]\Wlt 
 [ \alpha_{3}, \alpha_{8}, \alpha_{9}, \alpha_{11} ].
\end{align*}
We are now going to show the following, 
\begin{align*}
&a) [ \alpha_{1}, \alpha_{2}, \alpha_{6} ] \leadsto [\alpha_{3}, \alpha_{1}], \qquad 
b) [ \alpha_{1}, \alpha_{4}, \alpha_{7} ] \leadsto [\alpha_{3}, \alpha_{4}], \qquad 
c) [ \alpha_{2}, \alpha_{4}, \alpha_{7} ] \leadsto [\alpha_{3}, \alpha_{2}], \\
&d) [ \alpha_{1}, \alpha_{2}, \alpha_{4}, \alpha_{10} ]  \leadsto [\alpha_{1}, \alpha_{2}, \alpha_{6}],\qquad 
e) [ \alpha_{3}, \alpha_{8}, \alpha_{10}, \alpha_{11} ] \leadsto [\alpha_{5}, \alpha_{6}, \alpha_{10}], \\
&f) [ \alpha_{3}, \alpha_{9}, \alpha_{10}, \alpha_{11} ] \leadsto [\alpha_{5}, \alpha_{7}, \alpha_{10}], 
\end{align*}
and that every  $x_3(a_8a_9a_{10}/(\eta a_{11}^2))x_{8}(a_{8})x_{9}(a_{9})x_{10}(a_{10})x_{11}(a_{11})$, with 
$\eta \in \F_q^\times$ of trace $0$, is conjugated to some $x_{5}(\ba_{5})x_{6}(\ba_{6})x_{7}(\ba_{7})x_{10}(\ba_{10})$ with $\ba_{5}, \ba_{6}, \ba_{7}, \ba_{10} \in \F_q^\times$. 

For $a)$ we have that 
$$[\alpha_{1}, \alpha_{2}, \alpha_{6}]=[\alpha_{1}, \alpha_{6}, \alpha_{2}] \Wlt 
[\alpha_{3}, \alpha_{1}, \alpha_{8}] \leadsto_{x_6(a_8/a_1)} [\alpha_{3}, \alpha_{1}],$$
and similarly for $b)$ and $c)$. For $d)$, we have 
$$
[\alpha_{1}, \alpha_{2}, \alpha_{4}, \alpha_{10}] \Wlt [\alpha_{1}, \alpha_{2}, \alpha_{5}, \alpha_{12}] 
\leadsto_{x_9(a_1a_{12}/(a_2a_5))x_{10}(a_{12}/a_5)}
[\alpha_{1}, \alpha_{2}, \alpha_{5}] \Wlt [\alpha_{1}, \alpha_{2}, \alpha_{6}].
$$
For $e)$, we have that 
$$[\alpha_{3}, \alpha_{8}, \alpha_{10}, \alpha_{11}] \Wlt [\alpha_{5}, \alpha_{6}, \alpha_{10}, \alpha_{11}] 
\leadsto_{x_1(a_{11}/a_{10})x_2(a_6a_{11}/(a_5a_{10}))} 
[ \alpha_{5}, \alpha_{6}, \alpha_{10}].$$
A similar computation yields $f)$. 

Finally, we notice that the tuple $(\alpha_{3}, \alpha_{8}, \alpha_{9}, \alpha_{10}, \alpha_{11})$ 
is mapped to $(\alpha_{5}, \alpha_{6}, \alpha_{7}, \alpha_{11}, \alpha_{10})$ by a suitable element of $W$, hence if $a_8, a_9, a_{10}, a_{11} \in \F_q^\times$ then there exists 
$w \in W$ such that 
\begin{equation}\label{eq:4rhs}
(x_3(a_8a_9a_{10}/(\eta a_{11}^2))x_{8}(a_{8})\cdots x_{11}(a_{11}))^{\dot{w}}=x_5(\ba_5)\cdots x_{10}(\ba_{10})x_{11}(\ba_5\ba_{10}^2\eta/(\ba_6\ba_7)),
\end{equation}
where $\ba_6=a_8$, $\ba_7=a_9$, $\ba_{10}=a_{11}$, 
$\ba_5=a_8a_9a_{10}/(\eta a_{11}^2)$. 
If $t \in \F_q$, then we have that 
$$\left(x_5(\ba_5)x_6(\ba_6)x_7(\ba_7)x_{10}(\ba_{10})^{x}\right)=
x_5(\ba_5)x_6(\ba_6)x_7(\ba_7)x_{10}(\ba_{10})x_{11}(\ba_5\ba_6\ba_7t^2+\ba_5\ba_{10}t),$$
where $x=x_1(t)x_2(t\ba_6/\ba_5)x_4(t\ba_7/\ba_5)x_8(\ba_5\ba_6t)$. If we define $a:=\ba_5\ba_6\ba_7$, $b:=\ba_5\ba_{10}$ and $c:=\ba_5\ba_{10}^2\eta/(\ba_6\ba_7)$, then 
$\Tr(ac/(b^2))=\Tr(\eta)=0$, hence there exists $t \in \F_q$ such that $at^2+bt+c=0$. This means that 
the right hand side of Equation \eqref{eq:4rhs} is $U$-conjugate to $x_5(\ba_5)x_6(\ba_6)x_7(\ba_7)
x_{10}(\ba_{10})$. 

\subsection{The case $k=8$} Notice that $(\alpha_{5}, \alpha_{6}, \alpha_{7}, \alpha_{10}, \alpha_{11})$ 
is mapped to $(\alpha_{1}, \alpha_{2}, \alpha_{4}, \alpha_{10}, \alpha_{12})$ by $W$. Since 
$p=2$, we have that there exists some $w \in W$ such that 
$$(x_5(a_5)x_{6}(a_{6})x_7(a_7)x_{10}(a_{10})x_{11}(a_{11}))^{\dot{w}}=x_1(a_5)x_{2}(a_{6})x_4(a_7)x_{10}(a_{10})x_{12}(a_{11})$$
for every $a_5, a_6, a_7, a_{10}, a_{11} \in 
\F_q^{\times}$. Then the same argument and the 
same $w \in W$ as at the end of \S\ref{sub:k7} 
yield
$$(x_3(a_8a_9a_{10}/(\mu a_{11}^2))x_{8}(a_{8})\cdots x_{11}(a_{11}))^w=x_5(\ba_5)x_6(\ba_6)x_7(\ba_7)
x_{10}(\ba_{10})x_{11}(\ba_5\ba_{10}^2\mu/(\ba_6\ba_7))$$
for some $\ba_5, \ba_6, \ba_7$ and $\ba_{10}$ in $\F_q^\times$. 

\subsection{The cases $k=9, 10, 11$} We have that 
$$ [ \alpha_{3}, \alpha_{1}, \alpha_{2} ] \Wlt [ \alpha_{4}, \alpha_{5}, \alpha_{6} ]\Wlt [ \alpha_{4}, \alpha_{3}, \alpha_{8} ]
\Wlt [ \alpha_{7}, \alpha_{1}, \alpha_{2} ],$$
and the above triples all correspond to the simple roots of some root system of type $\rA_3$. 
By Proposition \ref{prop:typeA}, we have that $[ \alpha_{4}, \alpha_{3}, \alpha_{8} ] \leadsto 
[ \alpha_{3}, \alpha_{4}, \alpha_{8} ]$ and $[ \alpha_{7}, \alpha_{1}, \alpha_{2} ] \leadsto [\alpha_{1}, \alpha_{2}, \alpha_{7}]$. 
Finally, we notice that 
$$[\alpha_{1}, \alpha_{2}, \alpha_{4}, \alpha_{7}] \Wlt 
[\alpha_{1}, \alpha_{2}, \alpha_{7}, \alpha_{12}] \leadsto_{x_8(a_{12}/a_7)} 
[\alpha_{1}, \alpha_{2}, \alpha_{7}].$$
The cases $k=10$ and $k=11$ are again obtained by applying powers of $\tau$ 
to the case $k=9$, and by noticing that 
$$x_1(a_1)x_2(a_2)x_4(a_4)x_6(a_2a)x_7(a_4a)^u
=x_1(a_1)x_2(a_2)x_4(a_4)x_5(a_1a),$$
where $u=x_3(a)x_4(a_4)x_6(a_2a)x_{10}(a_2a_4a)$. 

\subsection{The case $k=12$}
Let $x:=x_1(a_1)x_2(a_2)x_{4}(a_{4})x_{6}(a_6^*)x_{7}(a_7^*)$, 
with $a_6^*, a_7^* \in \F_q^\times$ 
such that $(a_6^*, a_7^*)\ne a(a_2, a_4)$ for any $a \in \F_q^\times$. 
Then $[\alpha_{1}, \alpha_{2},\alpha_{ 4},\alpha_{ 6},\alpha_{ 7}] \Wlt 
[\alpha_{5},\alpha_{ 6},\alpha_{ 7}, \alpha_{2}, \alpha_{4}]$. Rewriting the product 
$x_5(a_1)x_6(a_2)x_{7}(a_{4})x_{2}(a_6^*)x_{4}(a_7^*) $, 
we see that it is equal to 
$$y:=x_2(a_6^*)x_4(a_7^*)x_5(a_1)x_6(a_2)x_7(a_4)
x_8(a_1a_6^*)x_9(a_1a_7^*)x_{10}(a^*)
x_{11}(a_1a_6^*a_7^*)x_{12}(a_1a^*),$$
with $a^*:=a_2a_7^*+a_4a_6^*$. Notice that $a^* \ne 0$ by assumption. 

We observe that 
$$x:=y^{x_3(a_2/a_6^*)}=x_2(a_6^*)x_4(a_7^*)
x_5(a_1)x_7(a^*/a_6^*)
x_8(b_8)x_9(b_9)x_{10}(b_{10})
x_{11}(b_{11})x_{12}(b_{12})$$
for some $b_8, \dots, b_{12} \in \F_q$. 
Now we notice that 
$$X_2X_4X_5X_7X_8X_9X_{10}X_{11}X_{12} 
\cong U(\rA_4(q))/Z(U(\rA_4(q)))$$
by extension to an isomorphism of the following map, 
$$x_4(t) \mapsto x_1(t), \qquad 
x_5(t) \mapsto x_2(t), \qquad 
x_2(t) \mapsto x_3(t), \qquad 
x_7(t) \mapsto x_4(t).$$
By applying exactly the same argument as in 
Proposition \ref{prop:typeA}, we see that $x$ is conjugate in $U$ to an element of the form  $x_2(\ba_2)x_4(\ba_4)
x_5(\ba_5)x_7(\ba_7)$ with 
$\ba_2, \ba_4, \ba_5, \ba_7 \in \F_q^\times$. But now we have that 
$$[\alpha_{2}, \alpha_{4}, \alpha_{5}, \alpha_{7}] \Wlt  [ \alpha_{2}, \alpha_{4}, \alpha_{3}, \alpha_{9} ] \Wlt 
[ \alpha_{6}, \alpha_{7}, \alpha_{1}, \alpha_{4} ] \Wlt [ \alpha_{3}, \alpha_{10}, \alpha_{1}, \alpha_{4} ].$$
By Proposition \ref{prop:typeA}, we have 
\begin{align*}
&[ \alpha_{2}, \alpha_{4}, \alpha_{3}, \alpha_{9} ] \leadsto [ \alpha_{3}, \alpha_{2}, \alpha_{4}, \alpha_{9} ], \qquad
[ \alpha_{6}, \alpha_{7}, \alpha_{1}, \alpha_{4} ] \leadsto [ \alpha_{1}, \alpha_{4}, \alpha_{6}, \alpha_{7} ],\\  
&[ \alpha_{3}, \alpha_{10}, \alpha_{1}, \alpha_{4} ] \leadsto [ \alpha_{3}, \alpha_{1}, \alpha_{4}, \alpha_{10} ],
\end{align*}
hence each element of row $4$ corresponding to 
$u_{12}^G$ in 
Table \ref{tab:fusiond4} is conjugate to suitable elements 
in rows $2$, $3$, $6$ and $7$. 

By conjugating $y$ by $x_3(a_4/a_7^*)$, and by noticing that 
$$X_2X_4X_5X_6X_8X_9X_{10}X_{11}X_{12} 
\cong U(\rA_4(q))/Z(U(\rA_4(q))),$$
we deduce as above that $x$ is conjugate to an element of the form  $x_2(\ba_2)x_4(\ba_4)
x_5(\ba_5)x_6(\ba_6)$ for some $\ba_2, \ba_4, \ba_5, \ba_6 \in \F_q^\times$. 
We apply the action of $W$, 
and again by Proposition \ref{prop:typeA} we get that 
\begin{align*}
&[\alpha_{2}, \alpha_{4}, \alpha_{5}, \alpha_{6}] \Wlt [\alpha_{6}, \alpha_{7}, \alpha_{1}, \alpha_{2}] 
\leadsto [\alpha_{1}, \alpha_{2}, \alpha_{6}, \alpha_{7}],\\
&[\alpha_{2}, \alpha_{4}, \alpha_{5}, \alpha_{6}] \Wlt [\alpha_{10}, \alpha_{3}, \alpha_{1}, \alpha_{2}] \leadsto 
[\alpha_{3}, \alpha_{1}, \alpha_{2}, \alpha_{10}],
\end{align*}
which means that $x$ is also conjugate 
to suitable elements in rows $1$ and $5$ of Table \ref{tab:fusiond4}.

\subsection{The cases $k=13, 14$} This follows from 
Table \ref{tab:reptorus}.

\section{Unitriangular shape of the $\ell$-decomposition matrices of $\SO_8^+(2^f)$} \label{sec:unitr}

We now focus on the determination of the $\ell$-decomposition matrices 
of $G$ where $\ell$ is an odd prime. 
By \cite{GH91} and \cite{Gec93}, we have that the unipotent characters form a 
basic set for the union of unipotent $\ell$-blocks. 
From now on, we just focus on the 
unipotent part of the decomposition matrix, which by slight 
abuse of terminology we refer to as the 
decomposition matrix itself of $G$. 

We now construct $\ell$-projective characters $\psi_1, \dots, \psi_{14}$ and $\psi_6', \psi_7'$, such that the 
matrix 
$$M:=\big( \langle \chi_i, \psi_j \rangle \big)_{i, j=1}^{14} $$
has a unitriangular 
shape, and $\psi_6'$ and $\psi_7'$ are useful to get upper bounds for some decomposition numbers of $G$. 
In particular, the $\ell$-decomposition matrices of $\SO_8^+(2^f)$ are also unitriangular. 
We explain in this section the construction and relevance of such projective characters, 
which are all obtained by inducing characters in $\Irr(U)$ to $G$. 

We first expand 
Equation \eqref{eq:isa} with 
$U$ in place of $H$ into a suitable form for our subsequent computations. 
We denote by $\fC$ the set of labels of the 
form $\cC$ for each family of conjugacy 
class representatives in the first column 
of Table \ref{tab:fusiond4}. For every $k \in \{1, \dots, 14\}$, we define 
$$\fC_k:=\{\cC \in \fC \mid \cC \subseteq 
u_k^G \cap U\},$$
and we define the \emph{family sum} 
$$\cS(\varphi, \cC):=\sum_{u \in \cC}\varphi(u)$$
with respect to $\varphi \in \Irr(U)$ and 
to the family $\cC$ of conjugacy class representatives of $U$. 

The following proposition shows that if we know all family sums 
for $\varphi$, then we know the multiplicities in $\varphi^G$ of each unipotent character of $G$. Notice 
that the value $d_{k, \cC}$ in the proposition is well-defined. 

\begin{prop}\label{prop:phiG}
Let $\chi$ be a unipotent character of $G$, and 
let $\varphi \in \Irr(U)$ and $\psi=\varphi^G$. For 
$k=1, \dots, 14$ and $\cC \in \fC_k$, let $\eta_k:=\overline{\chi(u_k)}$ and $d_{k, \cC}:=|C_U(u)|$ for $u \in \cC$. 
Then we have
$$\langle \psi, \chi \rangle =
  \sum_{k=1}^{14}\eta_k \left(\sum_{\cC \in \fC_k}
d_{k, \cC}^{-1}\cS(\varphi, \cC)\right).$$
\end{prop}

\begin{proof} We have that 
\begin{align*}
\langle \psi, \chi \rangle 
&=\frac{1}{|G|}\sum_{g \in G}\psi(g)\overline{\chi(g)} 
=\frac{1}{|G|}\sum_{u \in U^G}\psi(u)\overline{\chi(u)} 
=\sum_{k=1}^{14}\frac{\eta_k}{|C_G(u_k)|}\psi(u) \\
&=\sum_{k=1}^{14}\eta_k \left( \sum_{i=1}^{\ell(k)}\frac{\varphi(u_k^i)}{|C_U(u_k^i)|} \right)
=\sum_{k=1}^{14}\eta_k \left(\sum_{\cC \in \fC_k}
\frac{1}{d_{k, \cC}}\left(\sum_{u \in \cC}\varphi(u)\right) \right)\\
&
=\sum_{k=1}^{14}\eta_k \left(\sum_{\cC \in \fC_k}
d_{k, \cC}^{-1}\cS(\varphi, \cC)\right). \qedhere \end{align*}
\end{proof}

The projective characters $\psi_1, \psi_2, \dots, \psi_{14}$ and $\psi_6'$, $\psi_7'$ 
of $\SO_8^+(2^f)$ 
are constructed as follows. They are characters of the form $\varphi^G$ for suitable 
$\varphi \in \Irr(U)$ as described later. By Proposition 
\ref{prop:phiG}, the character values 
of $\varphi^G$ are obtained once we know both the 
fusion of the conjugacy classes of $U$ into $G$ determined in Section \ref{sec:fusio} and 
the character values of $\varphi$; 
the latter are provided in details in \cite{LM15} and 
\cite{GLM17} in the cases of our interest. In particular, each character $\varphi$ 
is inflated, then induced from a linear character of a certain abelian subquotient 
of $U$. 

The character labels for each $\varphi \in \Irr(U)$ of our interest are of the form 
$\chi_{\underline{i}}^{\underline{c}}$, in the notation of 
\cite[Table 4]{GLM17}, with $\underline{i}$ 
a multiset parametrizing the indices of root subgroups of 
the abelian subquotient of $U$ previously determined, and $\underline{c}$ 
a tuple in $\F_q$ that determines the values of the linear character that we inflate and induce from such a subquotient. 
We use the notation 
$\psi_{\underline{i}}^{\underline{c}}$ for $(\chi_{\underline{i}}^{\underline{c}})^G$ in the sequel.

We obtain the projective characters $\psi_6'$, $\psi_7'$ and $\psi_i$ with $i \in \{1, \dots, 5\} \cup \{10, \dots, 14\}$ 
by looking at 
a certain family 
$\cC \in \fC_k$, say $\cC:=\{x_{\alpha_{i_1}}(a_{i_1}) \cdots x_{\alpha_{i_m}}(a_{i_m}) \mid a_{i_1}, \dots, a_{i_m} \in \F_q^\times\}$, 
and by selecting a character in $\Irr(U)$ with label $\psi_{\underline{i}}^{\underline{c}}$, where $\underline{i}=(i_1, \dots, i_m)$ is a 
tuple indexed exactly by the indices in $\cC$, 
and $\underline{c}$ is a tuple which may depend on 
$a_{i_1}, \dots, a_{i_m}$ that we explicitly provide in 
each case. In most cases, we can select $\underline{c}$ 
to be 
$(1, \dots, 1)$. We point out that one can easily determine the columns of the $\ell$-decomposition matrices 
in Table \ref{tab:decn} corresponding to such $\ell$-projective characters, except for the character 
$\psi_{13}$, also by Harish-Chandra induction of projective covers of the Steinberg characters of Levi subgroups of $G$, as in \cite[Section 5]{GP92}. We include 
the computation for completeness, and to show 
that the results obtained by using the two methods agree. 

The construction of $\psi_{13}$ does deserve special attention. Namely this character is a natural replacement for the 
character $\Phi_{13}$ in \cite[Section 5]{GP92}, whose construction as a generalized Gelfand-Graev character 
of $\SO_8^+(p^f)$ for $p \ge 3$ cannot be extended to the case 
$p=2$, as $2$ is a bad prime in type $\rD_4$. 

The $\ell$-projective characters $\psi_6, \dots, \psi_9$ play a fundamental role to prove the unitriangularity of the 
$\ell$-decomposition matrices of $G$. These are constructed by inducing 
characters in $\Irr(U)$ whose degree is not a power of $q$, namely $q^3/2$; we recall 
on the other hand that if $r=p^f$ with $p \ge 3$, then the degree of an irreducible character of 
a Sylow $p$-subgroup of $\rD_4(r)$ is always a power of $r$. 
The importance of the construction of $\psi_6, \dots, \psi_9$ lies in the fact 
that such characters are obtained with substantially different 
methods from the case of $\SO_8^+(p^f)$ when $p$ is odd. 

We use the formulas in \cite[Table 4]{GLM17} in order to compute all family sums for our choices of characters. 
In the rest of this section we 
explain how to obtain such information. 
We use the notation $\ux(\ut)$ as in Section \ref{sec:preli}. 
Family sums involve many of the 
so-called character Gauss sums; the  
following computation is useful to deal with most cases. 

\begin{lem}\label{lem:Gauss} Let $k_1, \dots, k_m \in \F_q^{\times}$. Then 
$$\sum_{a_1, \dots, a_m \in \F_q^\times}\phi(k_1a_1+\dots+k_ma_m)=(-1)^m.$$
\end{lem}

\begin{proof} We recall that $\sum_{t \in \F_q}\phi(t)=0$, hence $\sum_{a \in \F_q^{\times}}\phi(a)=-1$. The 
claim then follows by induction on $m$. 
\end{proof}

\subsection{The characters $\psi_1, \dots, \psi_5, \psi_{10}, \dots, \psi_{12}, \psi_{14}$ and $\psi_6', \psi_7'$}\label{sub:114} We label the 
linear characters of $U$ as in \cite[\S4.1]{GLM17}, for $b_1, \dots, b_4 \in \F_q$. The character values are
$$\chi_{\text{lin}}^{b_1, b_2, b_3, b_4}(\ux(\ut))=\phi(b_1t_1+b_2t_2+b_3t_3+b_4t_4).$$
We choose $\ub=(b_1, b_2, b_3, b_4) \in \F_q^4$ for $\psi_1, \dots, \psi_5, \psi_{10}, \dots, \psi_{12}$ and $\psi_{14}$ 
as in Table \ref{tab:val14}, and 
$b=(1, 1, 0, 1)$ (respectively $b=(1, 0, 1, 0)$) for 
$\psi_6'$ (respectively $\psi_7'$). An application of Lemma \ref{lem:Gauss} easily gives the values of these characters. 
Notice that 
$\psi_4$ and $\psi_5$ (respectively $\psi_{11}$ and $\psi_{12}$) can be obtained by applying the triality automorphism 
to $\psi_3$ (respectively $\psi_{10}$).  

\begin{center}
\begin{table}[h]
\begin{scriptsize}
\begin{tabular}{|c|c|c|c|c|c|c|c|c|c|}
\hline
	 & $\psi_1$ & $\psi_2$ & $\psi_3$ & $\psi_4$ & $\psi_5$ & $\psi_{10}$ & $\psi_{11}$ & $\psi_{12}$ & $\psi_{14}$ \\
	 \hline
 $\ub$ & $(0, 0, 0, 0)$  & $(0, 0, 1, 0)$  & $(1, 1, 0, 0)$  & $(0, 1, 0, 1)$  & $(1, 0, 0, 1)$  & $(1, 1, 1, 0)$  & $(0, 1, 1, 1)$  & $(1, 0, 1, 1)$  & $(1, 1, 1, 1)$ \\
\hline
 \end{tabular}
\caption{The values of $\ub$ corresponding to the $\ell$-projective characters in \S\ref{sub:114}.}
\label{tab:val14}
\end{scriptsize}
\end{table}
\end{center}
\vspace{-12mm}

\subsection{The characters $\psi_6, \psi_7, \psi_8, \psi_9$} \label{sub:6789}
The 
irreducible characters of 
degree $q^3/2$ in $\Irr(U)$ are very important for our computations. 
Namely by inducing such 
characters, we obtain an identity block of size $4$ 
in Table \ref{tab:decn} whose rows are labelled by $\chi_6, \chi_7, \chi_8, \chi_9$, and whose columns are labelled by the $\ell$-projective characters $\psi_6, \psi_7, \psi_8, \psi_9$ that we now construct. 

The characters $\chi_{8, 9, 10, q^3/2}^{s, t, 1, 1, 1, 1}$, where $s, t \in \mu\F_2$ and $\mu$ is the previously fixed element of $\F_q$ of trace $1$, 
are defined as in \cite[Theorem 2.3]{LM15}. 
The value $\chi_{8, 9, 10, q^3/2}^{s, t, 1, 1, 1, 1}(\underline{x}(\underline{t}))$ is 
$$\begin{cases}
\frac{q^3}{2}\phi(st_1+t_7+\sum_{i=8}^{10}t_i), & \text{ if }\underline{x}(\underline{t}) \in Z, \\
\delta_{t_1, t_4}\delta_{t_2, t_4}\delta_{t_3, 1}(\delta_{t_1, 0}+\delta_{t_1, 1})\frac{q^2}{2}
\phi(st_1+tt_3+t_7+(t_5+t_7)(t_6+t_7)+\sum_{i=8}^{10}t_i), & \text{ otherwise, }\\
\end{cases}
$$
where 
$$Z=\{1, x_1(1)x_2(1)x_4(1)\} \cdot \{x_5(b)x_6(b)x_7(b) \mid b \in \F_q\} \cdot X_8X_9X_{10}X_{11}X_{12}.$$
The character labels corresponding to each $s$ and $t$ 
are as in Table \ref{tab:psi69}. 

\begin{center}
\begin{table}[h]
\begin{tabular}{|c|c|c|c|c|}
 \hline
	 & $\psi_6$ & $\psi_7$ & $\psi_8$ & $\psi_9$\\
	 \hline
 $\chi$ & $\chi_{8, 9, 10, q^3/2}^{\mu, \mu, 1, 1, 1, 1}$  & $\chi_{8, 9, 10, q^3/2}^{0, 0, 1, 1, 1, 1}$  & $\chi_{8, 9, 10, q^3/2}^{\mu, 0, 1, 1, 1, 1}$  & $\chi_{8, 9, 10, q^3/2}^{0, \mu, 1, 1, 1, 1}$ \\
\hline
 \end{tabular}
\caption{The character labels as in \S\ref{sub:6789} corresponding to $\psi_6, \psi_7, \psi_8$ and $\psi_9$.}
\label{tab:psi69}
\end{table}
\end{center}
\vspace{-10mm}

We compute the family sum $\cS(\chi_{8, 9, 10, q^3/2}^{s, t, 1, 1, 1, 1}, \cC_{3, q^8, 1}^{7})$. We get 
\begin{align*}
\cS(\chi_{8, 9, 10, q^3/2}^{s, t, 1, 1, 1, 1}, \cC_{3, q^8, 1}^{7})
&=\frac{q^2}{2} \sum_{\nu \in \F_q^\times \mid \Tr(\nu)=0} \sum_{a_8, a_9, a_{11} \in \F_q^\times}\phi(t+a_8+a_9+a_8^{-1}a_9^{-1}a_{11}^2\nu)\\
&=\frac{q^2}{2} \sum_{\nu \in \F_q^\times \mid \Tr(\nu)=0} \sum_{a_8, a_9, \omega_{11} \in \F_q^\times}\phi(t+a_8+a_9+a_8^{-1}a_9^{-1}\omega_{11}\nu)\\
&=-\frac{q^2}{2} \sum_{\nu \in \F_q^\times \mid \Tr(\nu)=0} \sum_{a_8, a_9 \in \F_q^\times}\phi(t+a_8+a_9)\\
&=-\frac{\phi(t)q^2(q-2)}{4},
\end{align*}
where we first used the fact that $\{a_{11}^2 \mid a_{11} \in \F_q^\times\}=\{\omega_{11} \mid \omega_{11} \in \F_q^\times\}$, and 
we then applied Lemma \ref{lem:Gauss}. The other family sums are obtained via standard applications of Lemma \ref{lem:Gauss}; notice that $\phi(\mu)=-1$, 
as $\Tr(\mu)=1$. 

\subsection{The character $\psi_{13}$} Let $d \ne 0, 1$, with $d \in \F_q^{\times}$; this certainly exists under our previous assumption 
$q>2$. We 
obtain the character $\psi_{13}$ by inducing to $G$ the character $\chi_{5, 6, 7}^{1, 1, 1, 0, 1, d} \in \Irr(U)$ whose 
values, as in \cite[Proposition 5.1]{GLM17}, are as follows, 
$$\chi_{5,6,7}^{1, 1, 1, 0, 1, d}(\underline{x}(\underline{t}))=q\delta_{(t_3, t_1+t_2+t_4)}\phi(t_2+d t_4+t_5+t_6+t_7).$$
Let $\chi:=\chi_{5,6,7}^{1, 1, 1, 0, 1, d}$ and $\cS:=\cS(\chi, \cC_{1, 2, 4, q^6}^{12, p=2})$. 
We have 
\begin{align*}
\cS & 
=q\sum_{a_1\in \F_q^\times}\sum_{a_2 \in \F_q^\times \setminus \{a_1\}} \sum_{a_6 \in \F_q^\times} 
\sum_{a_7 \in \F_q^\times \setminus \{a_6a_4/a_2\}}\phi(a_2+d(a_1+a_2)+a_6+a_7)\\
&=-q\sum_{a_1\in \F_q^\times}\sum_{a_2 \in \F_q^\times \setminus \{a_1\}} \sum_{a_6 \in \F_q^\times} 
\phi(a_2+d (a_1+a_2)+a_6)+\phi(a_2+d (a_1+a_2)+a_6(a_1/a_2))\\
&=2q\sum_{a_1\in \F_q^\times}\sum_{a_2 \in \F_q^\times \setminus \{a_1\}} \phi(d a_1 + (d+1)a_2))
=-2q\sum_{a_1\in \F_q^\times}\phi(d a_1)-2q \sum_{a_1\in \F_q^\times} \phi(a_1)\\
&=4q,
\end{align*}
where we used that $d \ne 0$ and $d + 1 \ne 0$. The values of 
$\cS(\chi, \cC_{1, 2, 4, q^6}^{i, p=2})$ for $i \in \{6, \dots, 11\}$ 
are obtained in a similar way. It is straightforward to obtain the values of all other family sums 
by using Lemma \ref{lem:Gauss}. 

By applying Proposition \ref{prop:phiG}, we obtain with the use of CHEVIE 
the matrix of the inner products 
of the unipotent characters $\chi_i$ 
with each $\psi_j$ or $\psi_6'$, 
$\psi_7'$ in Table \ref{tab:decn}. This gives an approximation of 
the $\ell$-decomposition matrices of $\SO_8^+(q)$. 
Notice that $\langle \chi_i, \psi_i\rangle = 1$, and $\langle \chi_i, 
\psi_j \rangle=0$ when $1\le i<j \le 14$. Hence we obtain the following result, which we state again. 

\begin{theo}\label{theo:uni}
Let $\ell \ne 2$ be a prime number. Then the restrictions of the  $\ell$-decomposition matrices 
of $\SO_8^+(2^f)$ to the set of unipotent characters have unitriangular shape. 
\end{theo}

Finally, we notice that the columns corresponding 
to every $\ell$-projective character in Table \ref{tab:decn}, except $\psi_6, \dots, \psi_9$ 
and $\psi_{13}$, do indeed give the same values as 
obtained in \cite[\S5]{GP92} with respect to the 
projectives $\Phi_1, \dots, \Phi_5$, $\Phi_7, \Phi_8$, $\Phi_{10}, \Phi_{11}, \Phi_{12}$ and 
$\Phi_{14}$, which are obtained by Harish-Chandra 
induction of projective covers of Steinberg characters of 
Levi subgroups. The labels for such Levi subgroups in terms of simple roots correspond to the labels of the characters 
of $\Irr(U)$ that we inflate and induce to 
obtain each of the $\psi_i$ or $\psi_i'$. 
Many of these $\ell$-projective characters of 
$G$ are not indecomposable. We will see in 
the next section that we obtain projective indecomposable 
characters of $G$ by inducing projective characters of such Levi subgroups which are not 
projective covers of Steinberg characters.

\begin{small}
\begin{center}
\begin{table}[t]
\scalebox{0.95}{
\begin{tabular}{|c|cccccccccccccccc|}
\hline
& $\psi_1$ & $\psi_2$ & $\psi_3$ & $\psi_4$ & $\psi_5$ & $\psi_6$ & $\psi_6'$ & $\psi_7$ & $\psi_7'$ & $\psi_8$& $\psi_9$ & $\psi_{10}$ & $\psi_{11}$ & $\psi_{12}$ & $\psi_{13}$ & $\psi_{14}$ \\
\hline
$\chi_1$
& $1$
& \multicolumn{1}{:c}{$\cdot$}
& $\cdot$
& $\cdot$
& $\cdot$
& $\cdot$
& $\cdot$
& $\cdot$
& $\cdot$
& $\cdot$
& $\cdot$
& $\cdot$
& $\cdot$
& $\cdot$
& $\cdot$
& $\cdot$
\\
\cdashline{1-3}
$\chi_2$ 
& $4$
& \multicolumn{1}{:c}{$1$}
& \multicolumn{1}{:c}{$\cdot$}
& $\cdot$
& $\cdot$
& $\cdot$
& $\cdot$
& $\cdot$
& $\cdot$
& $\cdot$
& $\cdot$
& $\cdot$
& $\cdot$
& $\cdot$
& $\cdot$
& $\cdot$
\\
\cdashline{1-1} \cdashline{3-6}
$\chi_3$ 
& $3$
& $1$
&\multicolumn{1}{:c}{$1$} 
& $\cdot$
& $\cdot$
& \multicolumn{1}{:c}{$\cdot$}
& $\cdot$
& $\cdot$
& $\cdot$
& $\cdot$
& $\cdot$
& $\cdot$
& $\cdot$
& $\cdot$
& $\cdot$
& $\cdot$
\\
$\chi_4$ 
& $3$
& $1$
&\multicolumn{1}{:c}{$0$} 
& $1$
& $\cdot$
& \multicolumn{1}{:c}{$\cdot$}
& $\cdot$
& $\cdot$
& $\cdot$
& $\cdot$
& $\cdot$
& $\cdot$
& $\cdot$
& $\cdot$
& $\cdot$
& $\cdot$
\\
$\chi_5$ 
& $3$
& $1$
&\multicolumn{1}{:c}{$0$} 
& $0$
& $1$
& \multicolumn{1}{:c}{$\cdot$}
& $\cdot$
& $\cdot$
& $\cdot$
& $\cdot$
& $\cdot$
& $\cdot$
& $\cdot$
& $\cdot$
& $\cdot$
& $\cdot$
\\
\cdashline{1-1} \cdashline{4-12}
$\chi_6$ 
& $2$
& $1$
& $1$ 
& $1$ 
& $1$ 
& \multicolumn{1}{:c}{$1$}
& $1$ 
& $\cdot$
& $\cdot$
& $\cdot$
& $\cdot$
& \multicolumn{1}{:c}{$\cdot$}
& $\cdot$
& $\cdot$
& $\cdot$
& $\cdot$
\\
$\chi_7$ 
& $6$
& $3$
& $1$
& $1$
& $1$
& \multicolumn{1}{:c}{$0$}
& $0$
& $1$
& $1$
& $\cdot$
& $\cdot$
& \multicolumn{1}{:c}{$\cdot$}
& $\cdot$
& $\cdot$
& $\cdot$
& $\cdot$
\\
$\chi_8$ 
& $ 8 $
& $4$
& $2$
& $2$
& $2$
& \multicolumn{1}{:c}{$0$}
& $1$
& $0$
& $1$
& $1$
& $\cdot$
& \multicolumn{1}{:c}{$\cdot$}
& $\cdot$
& $\cdot$
& $\cdot$
& $\cdot$
\\
$\chi_9$ 
& $0$
& $0$
& $0$
& $0$
& $0$
& \multicolumn{1}{:c}{$0$} 
& $0$
& $0$
& $0$
& $0$
& $1$
& \multicolumn{1}{:c}{$\cdot$} 
& $\cdot$
& $\cdot$
& $\cdot$
& $\cdot$
\\
\cdashline{1-1} \cdashline{7-15}
$\chi_{10}$ 
& $3$
& $2$
& $2$
& $1$
& $1$
& $\frac{q}{2}$
& $1$
& $\frac{q}{2}$
& $1$
& $\frac{q}{2}$
& $\frac{q}{2}$
& \multicolumn{1}{:c}{$1$}
& $\cdot$
& $\cdot$
& \multicolumn{1}{:c}{$\cdot$}
& $\cdot$
\\
$\chi_{11}$ 
& $3$
& $2$
& $1$
& $2$
& $1$
& $\frac{q}{2}$ 
& $1$
& $\frac{q}{2}$
& $1$
& $\frac{q}{2}$
& $\frac{q}{2}$
& \multicolumn{1}{:c}{$0$}
& $1$
& $\cdot$
& \multicolumn{1}{:c}{$\cdot$}
& $\cdot$
\\
$\chi_{12}$ 
& $3$
& $2$
& $1$ 
& $1$
& $2$
& $\frac{q}{2}$ 
& $1$
& $\frac{q}{2}$
& $1$
& $\frac{q}{2}$
& $\frac{q}{2}$
& \multicolumn{1}{:c}{$0$}
& $0$
& $1$
& \multicolumn{1}{:c}{$\cdot$}
& $\cdot$
\\
\cdashline{1-1} \cdashline{13-16}
$\chi_{13}$ 
& $4$
& $3$
& $2$
& $2$
& $2$
& $\frac{q^2-q}{2}$
& $1$
& $\frac{q^2-q}{2}$ 
& $2$
& $\frac{q^2-q}{2}$
& $\frac{q^2-q}{2}$
& $1$
& $1$
& $1$
& \multicolumn{1}{:c}{$1$}
& \multicolumn{1}{:c|}{$\cdot$}
\\
\cdashline{1-1} \cdashline{16-17}
$\chi_{14}$
& $1$
& $1$
& $1$
& $1$
& $1$
& $\frac{q^3}{2}$ 
& $1$
& $\frac{q^3}{2}$
& $1$
& $\frac{q^3}{2}$
& $\frac{q^3}{2}$
& $1$
& $1$
& $1$
& $q$
& \multicolumn{1}{:c|}{$1$}
\\
\hline
\end{tabular}
}
\caption{The unitriangular shape of the matrix $(\langle \chi_i, \psi_j\rangle)_{i, j}$ as in Section \ref{sec:unitr}, which implies the unitriangularity of the 
$\ell$-decomposition matrices of $G$ and certain upper bounds for the $\ell$-decomposition numbers.}
\label{tab:decn}
\end{table}
\end{center}
\end{small}
\vspace{-5mm}

\section{On the decomposition numbers of $\SO_8^+(2^f)$ when $\ell \mid q+1$}\label{sec:decnu}

The goal of this final section is to determine the $\ell$-decomposition 
matrix of $\SO_8^+(q)$ when $q=2^f$, and $\ell \ge 5$ is a prime number 
such that $\ell \mid q+1$ and $(q+1)_\ell > 5$. 
Not all methods used in the sequel can be applied when 
$(q+1)_\ell \in \{3, 5\}$, in which case one expects some entries of the $\ell$-decomposition matrices 
to be different from the ones in Table \ref{tab:finald}; see for instance \cite[Theorem 4.3]{HN14} and 
\cite[Remark 3.2]{DM16}. Notice that the assumption $\ell \ge 5$ clearly implies $q > 2$. 

We briefly explain as follows how to obtain 
the decomposition matrices in the remaining cases; the 
calculations are straightforward. 

\begin{itemize}
\item[-] If $\ell \mid q-1$, then the decomposition matrix of 
$G$ is determined in \cite[Section 8]{GH97b} by means of 
decomposition matrices of $q$-Schur algebras. 
\item[-] If $\ell \mid q^2+1$ and $(q^2+1)_\ell > 5$, the decomposition matrix of $G$ is 
obtained in \cite[Table 1]{DM16} by using 
the methods recalled later in this section 
and certain other techniques. As in \cite[Remark 3.2]{DM16}, 
some of the entries of the decomposition matrix in the case 
$(q^2+1)_\ell=5$ are different from those in \cite[Table 1]{DM16}. 
\item[-] If $\ell \nmid q \pm 1$ and $\ell \nmid q^2+1$, then the defect group of 
the principal $\ell$-block is cyclic, since it is a Sylow $\ell$-subgroup of $G$. In 
this case, the decomposition numbers of $G$ are encoded in its Brauer tree. Brauer trees have been 
determined in \cite{FS90} for the groups $\SO_8^+(p^f)$ when $p \ge 3$. 
It is a straightforward computation by using Harish-Chandra induction 
to verify that the Brauer trees are the same in the case 
of $p=2$; this is for instance clarified in \cite{CK17$^+$}. 
\end{itemize}

\vspace{-0.1mm}
In the rest of this work, we determine the $\ell$-decomposition numbers of $G$ when $\ell \mid q+1$, up to 
two parameters $\alpha, \beta \in \Z_{\ge 0}$ in the ninth column of Table \ref{tab:finald}. 
Our strategy is to exploit the construction of the $\ell$-projective characters $\psi_1, \dots, \psi_{14}$ and $\psi_6', \psi_7'$ 
in Section \ref{sec:unitr}, to obtain information on 
$\ell$-projective indecomposable characters $\Psi_1, \dots, \Psi_{14}$ of $G$. We say that 
each $\Psi_i$, $i=1, \dots, 14$ is a character of a PIM, which stands for projective indecomposable module. 

The 
methods we employ are among those outlined in 
\cite{DM16} and \cite[Section 4]{Dud17}. We collect and label such methods as follows. We denote by $R_w$ 
the virtual Deligne-Lusztig character $R_w(1)$ corresponding to $w \in W$ and $1=1_{\mathbf{T}^{\dot{w}F}} \in \Irr(\mathbf{T}^{\dot{w}F})$, and by $\ell(w)$ the 
length of $w$. 

\textbf{(Uni)} The unipotent part of the decomposition matrix of $G$ has lower 
unitriangular shape.

\textbf{(Sum)} If $\chi_1$ is a character of a PIM, and $\chi_1+\chi_2$ is 
a projective character, then $\chi_2$ is also a projective character. 

\textbf{(HC)} The Harish-Chandra induction and 
restriction of a projective character is 
again projective. 

\textbf{(Reg)} Let $w \in W$, and let us 
assume that there exists an $\ell$-character of 
$\mathbf{T}^{\dot{w}F}$ in general position.  
Then we have that 
$\langle \Psi, (-1)^{\ell(w)}R_w \rangle \ge 0$ for every projective character $\Psi$. 

\textbf{(DL)} If $\Psi$ is a projective 
character of $G$, and $w$ is minimal in the 
Bruhat order such that the unipotent part of 
$\Psi$ occurs in $R_w$, then $\langle 
\Psi, (-1)^{\ell(w)}R_w \rangle >0$. 

By \cite[Proposition 7.4]{DL76}, if a character 
$\theta$ is in general position then we 
have that $(-1)^{\ell(w)}R_w(\theta)$ is an 
irreducible character. Note that (Uni) is indeed satisfied by Theorem \ref{theo:uni}. 
(Sum) is a straightforward check. We recall that every $\ell$-projective module 
for $G$ is a direct sum of PIMs of $G$. 
The statement in (Reg) 
follows from \cite[Lemma 1.1]{Dud13}, while 
(DL) is a character-theoretical consequence of \cite[Section 8]{BR03}. 

Before we start the analysis, we recall 
some useful information about $G$. There are $13$ unipotent characters 
lying in the principal $\ell$-block 
when $\ell \mid q+1$; the character $\chi_8$ forms 
an $\ell$-block on its own. This can be readily checked 
for instance by applying \cite[Theorem A]{KM15}. 
Let $w_0$ be the longest element of $W$. Then 
$w_0F$ acts on $X(\bT)$ as $-q$, hence 
$\Irr(\mathbf{T}^{\dot{w_0}F})=X(\bT)/(w_0F-\text{id})X(\bT)
\cong (\Z/(q+1)\Z)^4$ and 
$|\mathbf{T}^{\dot{w_0}F}|=(q+1)^4$. 
As $\ell \ge 5$, we have that $\ell \nmid |W|$, and an $\ell$-character of $\mathbf{T}^{\dot{w_0}F}$ in general position corresponds to 
an irreducible character of $(X(\mathbf{T})/(q+1)X(\mathbf{T}))^4$ 
lying outside the reflection hyperplanes of 
$W=C_W(w_0F)$ in the reflection representation of $W$ 
on $X(\bT)/(w_0F-\text{id})X(\bT)$. Notice that if $(q+1)_\ell \le 5$ then 
such reflection hyperplanes cover the whole of $(\Z/(q+1)_\ell \Z)^4$. 
On the other hand, if 
$(q+1)_\ell > 5$ then one easily finds an $\ell$-element in $(\Z/(q+1) \Z)^4$; 
hence the assumptions of 
(Reg) are satisfied in this case. 

The program 
CHEVIE promptly gives the 
decomposition into unipotent characters 
of $R_w$ for every $w \in W$. In particular, 
$$R_{w_0}=\chi_1-4\chi_2+3\chi_3+3\chi_4+3\chi_5-6\chi_6-2\chi_7-8\chi_9+3\chi_{10}+3\chi_{11}+3\chi_{12}-4\chi_{13}+\chi_{14}.$$
If $\bL$ is an $F$-stable Levi subgroup of $\bG$, we denote by $R_{\bL}^{\bG}(\varphi)$ (respectively ${}^*R_{\bL}^{\bG}(\chi)$) the 
Harish-Chandra induction (respectively restriction) of $\varphi \in \Irr(L)$ (respectively $\chi \in \Irr(G)$). 
In particular, we denote by $\bL_{i_1, \dots, i_m}$ the Levi subgroup corresponding to the 
parabolic subgroup $\langle s_{i_1}, \dots, s_{i_m} \rangle$ of $W$. 
By slight abuse of notation, for a projective character $\Psi$ we also denote by 
$\Psi$ its unipotent part in the principal $\ell$-block of $G$. Of course by (Uni) 
we have that $\Psi_{14}=\psi_{14}$.  

\subsection{The projective indecomposables $\Psi_{10}, \Psi_{11}, \Psi_{12}$ and properties of $\Psi_{13}$} \label{sec:pr13}
For $i=10, 11, 12$, we have that 
$$\langle \psi_{i}, R_{w_0} \rangle=0, \,\, \text{and} \,\, \langle \chi_{i}+\chi_{13}, R_{w_0} \rangle\ne 0,  \,\,
\langle \chi_{i}+\chi_{14}, R_{w_0} \rangle \ne 0, \,\, \langle \chi_{13}+\chi_{14}, R_{w_0} \rangle \ne 0.$$
By (Reg), we have that $\psi_{i}$ is a character of a PIM for $i=10, 11, 12$, that is, $\Psi_i=\psi_i$. We apply 
(Reg) on $\Psi_{13}$. Write $\Psi_{13}=\chi_{13}+\delta\chi_{14}$. Then we have 
$$0 \le \langle \Psi_{13}, R_{w_0} \rangle = \langle \chi_{13}+\delta\chi_{14}, R_{w_0} \rangle 
=\langle \chi_{13}, R_{w_0} \rangle +\delta\langle \chi_{14}, R_{w_0} \rangle=-4+\delta,$$  
hence $\delta \ge 4$. By Table \ref{tab:decn}, we have that none of $\psi_6'$, $\psi_7'$ and  $\psi_i$ 
with $i \in \{1, \dots, 14\}\setminus\{9\}$ can be decomposed as $\Psi_{13}+\nu_{13}$ for some 
(projective) character $\nu_{13}$. 


\subsection{The projective indecomposables $\Psi_7$ and $\Psi_6$} Some upper bounds for the columns corresponding 
to $\Psi_7$ and $\Psi_6$ are given by the projective characters $\psi_7'$ and $\psi_6'$ as in 
Section \ref{sec:unitr}. We show here that these bounds are tight. In order to do this, we need some 
notation on the decomposition matrices of $\rA_3(q)$ as in \cite{Jam90}. The unipotent 
characters in type $\rA_{n-1}$ are parametrized by the partitions of $n$. Given a partition $\lambda=(\lambda_1^{n_1} \cdots
\lambda_s^{n_s})$ 
of $n$ with $\lambda_1 > \dots > \lambda_s$, we denote by $\rho_{\lambda}$ the corresponding unipotent character. Moreover, we denote 
by $\eta_1, \dots, \eta_5$ the characters of the PIMs in $\rA_3$ that label the 
columns of its $\ell$-decomposition matrix when 
$\ell \mid q+1$ in \cite[Appendix 1]{Jam90}, namely 
\begin{align*}
&\eta_1=\rho_{(4)}+\rho_{(31)}+\rho_{(21^2)}+\rho_{(1^4)}, \qquad \eta_2=\rho_{(31)}+\rho_{(2^2)}+\rho_{(21^2)},\\
&\eta_3=\rho_{(2^2)}+\rho_{(21^2)},\qquad
\eta_4=\rho_{(21^2)}+\rho_{(1^4)},\qquad
\eta_5=\rho_{(1^4)}.
\end{align*}

Notice that $\langle \psi_7', R_{w_0} \rangle = \langle \psi_6', R_{w_0} \rangle = 0$. We first examine $\psi_7'$. 
By (Uni), $\psi_7'$ is a non-negative integral combination of $\Psi_7$ and $\Psi_{10}, \dots, \Psi_{14}$. Suppose 
that $\psi_7'=\Psi_{10}+\nu_{10}^7$. 
Then $\nu_{10}^7=\chi_7+\chi_{11}+\chi_{12}+\chi_{13}$, and by (Sum) the character $\nu_{10}^7$ is projective. This is a contradiction, 
namely by (HC) we should have in this case that ${}^*R_{\bL_{1, 3, 4}}^{\bG}(\nu_{10}^7)$ is also a projective character, but 
$${}^*R_{\bL_{1, 3, 4}}^{\bG}(\nu_{10}^7)=\rho_{(31)}+4\rho_{(21^2)}+\rho_{(1^4)}=\eta_2-\eta_3+4\eta_4-3\eta_5,$$
which is a combination of 
$\eta_1, \dots, \eta_5$ 
also with negative coefficients. Similarly, 
we deduce that $\psi_7'=\Psi_{i}+\nu_i^7$ with $i \in \{11, 12\}$ cannot happen. By \S\ref{sec:pr13}, we 
have that $\Psi_{13}$ is not a summand of $\psi_7'$, and now we apply (Reg) to see that $\Psi_{14}$ cannot 
be either. Hence $\Psi_7=\psi_7'$ is a character of a PIM. 

A similar argument applies for the character $\psi_6'$. We just show that one cannot write $\psi_6'=\Psi_{10}+\nu_{10}^6$. 
Namely if this were the case instead, again by (Sum) the character $\nu_{10}^6=\chi_6+\chi_{11}+\chi_{12}$ would be projective, but we have 
that 
$${}^*R_{\bL_{1, 3, 4}}^{\bG}(\nu_{10}^6)=\rho_{(2^2)}+2\rho_{(21^2)}=\eta_3+\eta_4-\eta_5,$$
which leads to a contradiction. The character $\Psi_6=\psi_6'$ is then projective indecomposable. 

\subsection{The projective indecomposables $\Psi_{3}, \Psi_4$ and $\Psi_5$} We determine the projective indecomposable character $\Psi_5$; for $\Psi_4$ and $\Psi_3$ we apply the triality 
automorphism $\tau$ to get the same conclusion. 

By (HC), we have that 
$$\psi_5':=R_{\bL_{1, 3, 4}}^\bG(\eta_3)
=\chi_3+\chi_6+\chi_7+\chi_{10}+\chi_{11}
+\chi_{12}+\chi_{13}$$
is a projective character. Since $\langle 
\psi_5', \chi_{14} \rangle =0$, by (Uni) and 
the construction of $\Psi_6, \Psi_7$ and 
$\Psi_{10}, \dots, \Psi_{14}$ it is easy to 
see that $\psi_5'$ is indecomposable, that is, 
$\Psi_5=\psi_5'$. Similarly, we get 
$\Psi_4=\psi_4'$ and $\Psi_3=\psi_3'$.

\subsection{The projective indecomposable $\Psi_2$} Let us consider
$$\psi_2':=R_{\bL_{1,2,3}}^{\bG}(\eta_2)=\chi_2+\chi_3+\chi_4+\chi_5+\chi_6+2\chi_7+\chi_{10} 
+\chi_{11}+\chi_{12}+\chi_{13}.$$
By (HC), $\psi_2'$ is a projective character of $G$. We now claim that $\Psi_2=\psi_2'$.  
By looking at $\langle \chi_{14}, \Psi_i \rangle$ for $i \ge 3$, we have that either $\psi_2'$ is 
projective, or $\psi_2'=\Psi_i+\nu_i^2$ with $i \in \{3, 4, 5, 14\}$ and $\nu_i^2$ projective by (Sum). 
Let us for example assume $i=3$. Then $\nu_3^2=\chi_2+\chi_4+\chi_5+\chi_7$, and 
$${}^*R_{\bL_{1, 2, 3}}^{\bG}(\nu_3^2)=\rho_{(4)}+4\rho_{(31)}+\rho_{(21^2)}=\eta_1+3\eta_2-3\eta_3-\eta_5,$$
which leads to a contradiction. The same argument applies in the case $i\in \{4, 5\}$. Since 
$\langle \psi_2', R_{w_0} \rangle=0$, 
by (Reg) we have that $\psi_2'=\Psi_{14}+\nu_{14}^2$ cannot occur. 
The claim follows. 

\subsection{The projective indecomposable $\Psi_1$} We now consider 
$$\psi_1':=R_{\bL_{1,2,3}}^{\bG}(\eta_1)
=\chi_1+2\chi_2+\chi_3+\chi_4
+\chi_5+2\chi_7+
\chi_{10}+\chi_{11}+\chi_{12}+2\chi_{13}
+\chi_{14},$$
which is projective by (HC). Let us suppose 
that $\psi_1'$ is decomposable. In this 
case, by (Uni) and the fact that $\langle \psi_1', 
\chi_6\rangle=0$, we would have that 
$\psi_1'=\Psi_i+\nu_i^1$ for some 
$i \in \{7, 10, 11, 12, 14\}$, with 
$\nu_i^1$ projective by (Sum). 

Assume first that $i=7$. Then $\nu_7^1=\chi_1+2\chi_2+\chi_3+\chi_4+\chi_5+\chi_7$, and 
$${}^*R_{\bL_{1, 2, 3}}^\bG(\nu_6^1)=
4\rho_{(4)}+5\rho_{(31)}+\rho_{(2^2)}+\rho_{(21^2)}=
4\eta_1+\eta_2-4\eta_4,$$
which by (HC) contradicts the projectivity of $\nu_6^1$. If 
$i=10$, then we have that $\nu_{10}^1=\chi_1+2\chi_2+\chi_3+\chi_4+\chi_5+2\chi_7
+\chi_{11}+\chi_{12}+\chi_{13}$, and 
$${}^*R_{\bL_{1, 3, 4}}^\bG(\nu_{10}^1)=
4\rho_{(4)}+6\rho_{(31)}+\rho_{(2^2)}+5\rho_{(21^2)}+\rho_{(1^4)}=
4\eta_1+2\eta_2-\eta_3-3\eta_5,$$
again a contradiction. The same argument holds when 
$i=11$ or $i=12$ by applying $\tau$. We then notice 
that $\langle \psi_1', R_{w_0} \rangle=0$ and by (Reg) 
we deduce that $\psi_1'=\Psi_{14}+\nu_{14}^1$ cannot hold. Hence 
$\Psi_1=\psi_1'$ is a character of a PIM. 

\subsection{The projective indecomposable $\Psi_{13}$ 
and properties of $\Psi_9$} Finally, to obtain information 
on $\Psi_{13}$ and on the character $\Psi_9$ that corresponds to 
the cuspidal character $\chi_9$, we use (DL). We manage to 
show that the lower bound previously obtained for the 
constant $\delta$ is tight, and to determine the decomposition 
of $\Psi_9$ in terms of two non-negative 
parameters $\alpha$ and $\beta$. 
The determination of $\alpha$ and $\beta$, or 
even a universal bound for them, seems to be 
more complicated to obtain, as in the case of $\SO_8^+(p^f)$ 
with $p \ge 3$ where similar parameters show up; this is beyond 
the methods of this work.  
We can still give linear upper bounds in $q$ for $\alpha$ and $\beta$ 
by Table \ref{tab:decn} and by using the inequality $
\beta \le 3\alpha-1$ of the following computations; 
in this way we have 
$\alpha \le q/2$, $\beta 
\le (3q-2)/2$ and $-9\alpha+4\beta+8 \ge 0$. 

We first write $\Psi_{13}=\chi_{13}+\delta \chi_{14}$, with 
$\delta \ge 4$ by \S\ref{sec:pr13}, and 
$$\Psi_9=\chi_9+\alpha\chi_{10}+\alpha\chi_{11}+\alpha\chi_{12}
+\beta\chi_{13}+\gamma\chi_{14}.$$
Here we used that $\Psi_9$ is fixed by $\tau$ and that 
$\chi_{10}$, $\chi_{11}$ and $\chi_{12}$ are permuted 
by $\tau$, to deduce that $\langle \Psi_9, \chi_{10} \rangle=
\langle \Psi_9, \chi_{11} \rangle=\langle \Psi_9, \chi_{12} \rangle$. We apply (Reg) and we obtain 
\begin{equation}\label{eq:gamma1}
0 \le \langle \Psi_9, R_{w_0} \rangle = -8+9\alpha-4\beta+\gamma, 
\end{equation}
hence $\gamma \ge -9\alpha+4\beta+8$. 

It is straightforward to produce a program in CHEVIE that 
for each $w \in W$ decomposes 
$$R_w=\sum_{i=1}^{14} a_{w, i} \Psi_i$$
for some $a_{w, i} \in \Z_{\ge 0}$ 
which depend on the parameters $\alpha$, $\beta$, 
$\gamma$ and $\delta$. 
In the above decomposition, the characters $\Psi_{13}$ and 
$\Psi_{14}$ possibly occur on $R_c$ for an arbitrary Coxeter element $c$ of $W$, 
with 
$$ 
a_{c, 13}=
3\alpha-\beta-1 \qquad \text{and} \qquad 
a_{c, 14}=
-\delta(3\alpha-\beta-1) 
+3\alpha-\gamma+4,$$
and we have that $ a_{w, 13} = 
a_{w, 14}=0$ for every 
$w \lneq c$ in the Bruhat order. 
As $\ell(c)=4$, by (DL) we have that $3\alpha-\beta-1 \ge 0$, and that 
\begin{equation}\label{eq:gamma2}-\delta(3\alpha-\beta-1) 
+3\alpha-\gamma+4 \ge 0.
\end{equation}
By combining Equations \eqref{eq:gamma1} and 
\eqref{eq:gamma2}, and recalling that $\delta \ge 4$, we get
$$-9\alpha+4\beta+8\le \gamma \le 
-\delta(3\alpha-\beta-1)+3\alpha+4 \le 
-4(3\alpha-\beta-1)+3\alpha+4 =
-9\alpha+4\beta + 8,$$
hence $\gamma=-9\alpha+4\beta + 8$. By substituting 
$\gamma$ into Equation \eqref{eq:gamma2} one gets 
\begin{equation}\label{eq:delta}
-\delta(3\alpha-\beta-1)+4(3\alpha-\beta-1) \ge 0 
\Longrightarrow (\delta-4)(3\alpha-\beta-1) \le 0.
\end{equation}

Let us first assume that $3\alpha-\beta-1 \ge 1$. 
Then by Equation \eqref{eq:delta} we have that $\delta \le 4$, hence $\delta=4$. In this case, we have that $a_{c, 14}=0$. We 
then check that the only element $w \in W$ such that 
$a_{w, 14} \ne 0$ is $w_0$, with 
$a_{w_0, 14}=192$. Hence (DL) does not give 
more information on $\alpha$ and $\beta$. 

Finally, let us now assume that $3\alpha-\beta-1=0$, 
that is, $\beta=3\alpha-1$. Then 
$\gamma=3\alpha+4$, and $a_{c, 13}
=a_{c, 14}=0$. Moreover, we have that 
$a_{w, 13}=a_{w, 14}=0$ for every $w$ of length $5$ 
or less. 
The element $w'=s_1s_2s_3s_1s_4s_3$ 
is minimal in the Bruhat order 
such that $a_{w', 13}$ and $a_{w', 14}$ can possibly 
be nonzero; one has 
$$a_{w', 13}=4 \qquad \text{and} \qquad 
a_{w', 14}=-4\delta+16.$$
Since $\ell(w')=6$, by (DL) we have that 
$-4\delta+16\ge 0$, hence $\delta \le 4$, which 
implies $\delta=4$. In this case we have $a_{w, 14}
\ne 0$ for every $w \in W \setminus\{w_0\}$, and 
$a_{w_0, 14}=192$. The method (DL) cannot be applied 
further to get more information on $\alpha$.

\section*{Appendix: computations of the fusion of $U$-conjugacy classes into $G$}

We collect in this appendix 
a table which contains the information on the fusion of the conjugacy classes 
of $U$ into $G$, namely Table \ref{tab:fusiond4}. The first column labels the families of $U$-conjugacy class 
representatives as in \cite[Section 3]{GLM17}, with expressions given in the second column 
as products of root elements; each element is in the unipotent conjugacy class 
$u_k^G$ in the third column. Recall that $\mu$ is a fixed element in $\F_q^\times$ of trace $1$. 

\begin{center}
\begin{small}
 \begin{longtabu}{|c|c|c|} 
 \hline
Label & Representative in $U$   & Fusion in $G$ \\
\hline
\hline
\endhead
\hline
\endfoot
\hline \caption{The fusion of the conjugacy classes of $U$ into $G$.} \label{tab:fusiond4}
\endlastfoot
\label{tab:long}
$\cC^{1}$  & $x_{12}(0) $ &     $u_{1}^G$ \\
\hline
\hline
$\cC_{i}^{2}, i=1, \dots, 12$ & $x_{i}(a_{i}) $ & 
$u_{2}^G$ \\
\hline
\hline
$\cC_{1, 2, q^8}^{3}$  & $x_1(a_1)x_2(a_2) $ &     
\multirow{5}{*}{$ u_{3}^G$} \\
\cline{1-2}
$\cC_{4, q^8}^{3}$  & $x_4(a_4)x_{12}(a_{12}) $ &      \\
\cline{1-2}
$\cC_{3, q^8}^{3}$  & $x_3(a_3)x_8(a_8) $ &      \\
\cline{1-2}
$\cC_{5, 6, q^9}^{3}$  & $x_5(a_5)x_6(a_6) $ &      \\
\cline{1-2}
$\cC_{7, q^9}^{3}$  & $x_7(a_7)x_{11}(a_{11}) $ &      \\
\cline{1-2}
$\cC_{8, 9, 10, q^{10}}^{3}$  & $x_9(a_9)x_{10}(a_{10}) $ &      \\
\hline
\hline
$\cC_{1, 4, q^8}^{4}$  & $x_1(a_1)x_4(a_4) $ &    
\multirow{5}{*}{$ u_{4}^G$} \\
 \cline{1-2}
 $\cC_{2, q^8}^{4}$  & $x_2(a_2)x_{12}(a_{12}) $ &      \\ 
 \cline{1-2}
 $\cC_{3, q^8}^{4}$  & $x_3(a_3)x_9(a_9) $ &      \\
\cline{1-2}
 $\cC_{5, 7, q^9}^{4}$  & $x_5(a_5)x_7(a_7) $ &      \\
 \cline{1-2}
 $\cC_{6, q^9}^{4}$  & $x_6(a_6)x_{11}(a_{11}) $ &      \\
 \cline{1-2}
 $\cC_{8, 9, 10, q^{10}}^{4}$  & $x_8(a_8)x_{10}(a_{10}) $ &      \\
 \hline
 \hline
$\cC_{2, 4, q^8}^{5}$  & $x_2(a_2)x_4(a_4) $ &     
\multirow{5}{*}{$ u_{5}^G$} \\
 \cline{1-2}
 $\cC_{1, q^8}^{5}$  & $x_1(a_1)x_{12}(a_{12}) $ &      \\
 \cline{1-2}
 $\cC_{3, q^8}^{5}$  & $x_3(a_3)x_{10}(a_{10}) $ &      \\
\cline{1-2}
 $\cC_{6, 7, q^9}^{5}$  & $x_6(a_6)x_7(a_7) $ &      \\
 \cline{1-2}
 $\cC_{5, q^9}^{5}$  & $x_5(a_5)x_{11}(a_{11}) $ &      \\
 \cline{1-2}
 $\cC_{8, 9, 10, q^{10}}^{5}$  & $x_8(a_8)x_9(a_9) $ &      \\
 \hline
 \hline
$\cC_{1, 2, 4, q^7}^{6, p=2}$  &  $x_1(a_1)x_2(a_2)x_{4}(a_{4}) $ &     \multirow{13}{*}{$u_6^G$} \\
 \cline{1-2}
$\cC_{1, 2, q^8}^{6}$  & $x_1(a_1)x_2(a_2)x_{12}(a_{12}) $ &      \\
 \cline{1-2}
$\cC_{1, 4, q^8}^{6}$  & $x_1(a_1)x_4(a_4)x_{12}(a_{12}) $ &      \\
 \cline{1-2}
$\cC_{2, 4, q^8}^{6}$  & $x_2(a_2)x_4(a_4)x_{12}(a_{12}) $ &     \\
 \cline{1-2}
$\cC_{5, 6, 7, q^8}^{6}$   & $x_5(a_5)x_{6}(a_{6})x_7(a_7)$ &    \\
 \cline{1-2}
$\cC_{5, 6, q^9}^{6}$  & $x_5(a_5)x_6(a_6)x_{11}(a_{11}) $ &      \\
 \cline{1-2}
$\cC_{5, 7, q^9}^{6}$  & $x_5(a_5)x_7(a_7)x_{11}(a_{11}) $ &      \\
 \cline{1-2}
$\cC_{6, 7, q^9}^{6}$  & $x_6(a_6)x_7(a_7)x_{11}(a_{11}) $ &      \\
 \cline{1-2}
$\cC_{8, 9, 10, q^{10}}^{6}$   & $x_8(a_8)x_{9}(a_{9}) x_{10}(a_{10})$ &    \\
 \cline{1-2}
$\cC_{3, q^8, 1}^{6}$   & $x_3(a_3)x_8(a_8)x_{9}(a_{9}) $ &    \\
 \cline{1-2}
$\cC_{3, q^8, 2}^{6}$   & $x_3(a_3)x_8(a_8)x_{10}(a_{10}) $ &    \\
 \cline{1-2}
$\cC_{3, q^8, 3}^{6}$   & $x_3(a_3)x_9(a_9)x_{10}(a_{10}) $ &    \\
 \cline{1-2}
$\cC_{3, q^8, 4}^{6}$   & $x_3(a_3)x_8(a_8)x_9(a_9)x_{10}(a_{10}) $ &    \\
\hline
\hline
$\cC_{1,3, q^5}^{7}$  & $x_3(a_3)x_1(a_1)$ &     \multirow{28}{*}{$ u_{7}^G$} \\
 \cline{1-2}
$\cC_{2,3, q^5}^{7}$  & $x_3(a_3)x_2(a_2)$ &      \\
 \cline{1-2}
$\cC_{3, 4, q^5}^{7}$  & $x_3(a_3)x_4(a_4)$ &      \\
 \cline{1-2}
 $\cC_{1, 2, 4, 2q^7}^{7, p=2}$  &  $x_1(a_1)x_2(a_2)x_{4}(a_{4})x_{10}(a_{10}) $
 &    \\
\cline{1-2}
 $\cC_{1, 2, q^6}^{7}$ & $x_1(a_1)x_2(a_2)x_{6}(a_{6}) $ &  \\
 \cline{1-2}
$\cC_{1, 2, q^7}^{7}$  & $x_1(a_1)x_2(a_2)x_{10}(a_{10})$ &       \\
\cline{1-2}
 $\cC_{1, 4, q^6}^{7}$ & $x_1(a_1)x_4(a_4)x_{7}(a_{7}) $ &  \\
 \cline{1-2}
$\cC_{1, 4, q^7}^{7}$  & $x_1(a_1)x_4(a_4)x_{10}(a_{10})$ &     \\
\cline{1-2}
 $\cC_{2, 4, q^6}^{7}$ & $x_2(a_2)x_4(a_4)x_{7}(a_{7}) $ &  \\
\cline{1-2}
 $\cC_{2, 4, q^7}^{7}$  & $x_2(a_2)x_4(a_4)x_{9}(a_{9})$ &      \\
 \cline{1-2}
$\cC_{1, q^6, 1}^{7}$  & $x_1(a_1)x_6(a_6)$ &      \\
\cline{1-2}
$\cC_{1, q^6, 2}^{7}$  & $x_1(a_1)x_7(a_7)$ &      \\
  \cline{1-2}
$\cC_{1, q^7}^{7}$  & $x_1(a_1)x_{10}(a_{10})$ &      \\
\cline{1-2}
$\cC_{2, q^6, 1}^{7}$  & $x_2(a_2)x_5(a_5)$ &      \\
\cline{1-2}
$\cC_{2, q^6, 2}^{7}$  & $x_2(a_2)x_7(a_7)$ &      \\
  \cline{1-2}
$\cC_{2, q^7}^{7}$  & $x_2(a_2)x_{9}(a_{9})$ &      \\
\cline{1-2}
$\cC_{4, q^6, 1}^{7}$  & $x_4(a_4)x_5(a_5)$ &      \\
\cline{1-2}
$\cC_{4, q^6, 2}^{7}$  & $x_4(a_4)x_6(a_6)$ &      \\
  \cline{1-2}
$\cC_{4, q^7}^{7}$  & $x_4(a_4)x_{8}(a_{8})$ &      \\
\cline{1-2}
 \multirow{2}{*}{$\cC_{3, q^8, 1}^{7}$}  & $x_3(d_1)x_{8}(a_{8})x_{9}(a_{9})x_{10}(a_{10})x_{11}(a_{11})$, &      \\
  & $\Tr(a_8a_9a_{10}/(d_1a_{11}^2))=0$ & \\ 
\cline{1-2}
$\cC_{3, q^8, i}^{7}, i=2, \dots, 8$  & $x_3(a_3)x_{8}(f_{8})x_{9}(f_{9})x_{10}(f_{10})x_{11}(a_{11}), f_8f_9f_{10}=0$ &      \\
\cline{1-2}
$\cC_{5, 6, 7, 2q^8}^{7}$  & $x_{5}(a_{5})x_{6}(a_{6})x_{7}(a_{7})x_{10}(a_{10})$ &      \\
\cline{1-2}
$\cC_{5, 6, q^8}^{7}$  & $x_5(a_5)x_{6}(a_{6})x_{10}(a_{10})$ &    \\
\cline{1-2}
$\cC_{5, 7, q^8}^{7}$  & $x_5(a_5)x_{7}(a_{7})x_{10}(a_{10})$ &    \\
\cline{1-2}
$\cC_{6, 7, q^8}^{7}$  & $x_6(a_6)x_{7}(a_{7})x_{9}(a_{9})$ &    \\
\cline{1-2}
$\cC_{5, q^8}^{7}$   & $x_5(a_5)x_{10}(a_{10})$ &    \\
\cline{1-2}
$\cC_{6, q^8}^{7}$   & $x_6(a_6)x_{9}(a_{9})$ &    \\
\cline{1-2}
$\cC_{7, q^8}^{7}$   & $x_7(a_7)x_{8}(a_{8})$ &    \\
\hline
\hline
$\cC_{1, 2, 4, 2q^7}^{8, p=2}$  &  $x_1(a_1)x_2(a_2)x_{4}(a_{4})x_{10}(a_{10})x_{12}(a_1a_{10}^2\mu/a_2a_4) $ 
&      \multirow{3}{*}{$u_8^G$} \\
\cline{1-2}
 \multirow{2}{*}{$\cC_{3, q^8}^{8}$}  & $x_3(d_2)x_{8}(a_{8})x_{9}(a_{9})x_{10}(a_{10})x_{11}(a_{11})$, &      \\
  & $\Tr(a_8a_9a_{10}/(d_2a_{11}^2))=1$ & \\
\cline{1-2} 
\cline{1-2}
$\cC_{5, 6, 7, 2q^8}^{8}$   & $x_5(a_5)x_{6}(a_{6})x_7(a_7)x_{10}(a_{10})x_{11}(a_5a_{10}^2\mu/a_6a_7)$ &    \\
\hline
\hline
$\cC_{1,2,3, q^5}^{9}$ & $x_3(a_3)x_1(a_1)x_2(a_2)$ & \multirow{5}{*}{$ u_{9}^G$}\\
\cline{1-2}
$\cC_{3, 4, q^5}^{9}$  & $x_3(a_3)x_4(a_4)x_{8}(a_{8})$ &      \\
\cline{1-2}
$\cC_{1,2,4,q^6}^{9, p=2}$ & $x_1(a_1)x_2(a_2)x_{4}(a_{4})x_{7}(a_{7})$ &  \\
\cline{1-2}
$\cC_{1, 2, q^6}^{9}$  & $x_1(a_1)x_2(a_2)x_{7}(a_{7}) $ & \\
\cline{1-2}
$\cC_{4, q^6}^{9}$  & $x_4(a_4)x_5(a_5)x_6(a_6) $ & \\
\hline
\hline
$\cC_{1,3,4, q^5}^{10}$ & $x_3(a_3)x_1(a_1)x_4(a_4)$ & \multirow{5}{*}{$ u_{10}^G$}\\
\cline{1-2}
$\cC_{2, 3, q^5}^{10}$  & $x_3(a_3)x_2(a_2)x_{9}(a_{9})$ &      \\
\cline{1-2}
$\cC_{1,2,4,q^6}^{10, p=2}$ & $x_1(a_1)x_2(a_2)x_{4}(a_{4})x_{6}(a_{6})$ &  \\
\cline{1-2}
$\cC_{1, 4, q^6}^{10}$  & $x_1(a_1)x_4(a_4)x_{6}(a_{6}) $ & \\
\cline{1-2}
$\cC_{2, q^6}^{10}$  & $x_2(a_2)x_5(a_5)x_7(a_7) $ & \\
\hline
\hline
$\cC_{2,3,4, q^5}^{11}$ & $x_3(a_3)x_2(a_2)x_4(a_4)$ & \multirow{5}{*}{$ u_{11}^G$}\\
\cline{1-2}
$\cC_{1, 3, q^5}^{11}$  & $x_3(a_3)x_1(a_1)x_{10}(a_{10})$ &      \\
\cline{1-2}
$\cC_{1,2,4,q^6}^{11, p=2}$ & $x_1(a_1)x_2(a_2)x_{4}(a_{4})x_{6}(a_{2}a)x_{7}(a_{4}a)$ &  \\
\cline{1-2}
$\cC_{2, 4, q^6}^{11}$  & $x_2(a_2)x_4(a_4)x_{5}(a_{5}) $ & \\
\cline{1-2}
$\cC_{1, q^6}^{11}$  & $x_1(a_1)x_6(a_6)x_7(a_7) $ & \\
\hline
\hline 
$\cC_{1,2,3, q^5}^{12}$ & $x_3(a_3)x_1(a_1)x_2(a_2)x_{10}(a_{10})$ & 
    \multirow{7}{*}{$u_{12}^G$} \\
\cline{1-2}
 $\cC_{1,3,4, q^5}^{12}$ & $x_3(a_3)x_1(a_1)x_4(a_4)x_{10}(a_{10})$
   &   \\
\cline{1-2}
$\cC_{2,3,4, q^5}^{12}$ & $x_3(a_3)x_2(a_2)x_4(a_4)x_{9}(a_{9})$ & 
     \\
\cline{1-2}
$\cC_{1, 2, 4, q^6}^{12, p=2}$  &  $x_1(a_1)x_2(a_2)x_{4}(a_{4})x_{6}(a_6^*)x_{7}(a_7^*) $ 
&     \\ 
\cline{1-2}
$\cC_{1, 2, q^6}^{12}$  & $x_1(a_1)x_2(a_2)x_{6}(a_{6})x_{7}(a_{7}) $ &      \\
\cline{1-2}
$\cC_{1, 4, q^6}^{12}$  & $x_1(a_1)x_4(a_4)x_{6}(a_{6})x_{7}(a_{7}) $ &     \\
\cline{1-2}
$\cC_{2, 4, q^6}^{12}$  & $x_2(a_2)x_4(a_4)x_{5}(a_{5})x_{7}(a_{7}) $ &    \\ 
\hline
\hline
  $\cC_{1,2,3,4}^{13}$ &  $x_1(a_1)x_2(a_2)x_3(a_3)x_4(a_4)$  & $ u_{13}^G$ \\
\hline
\hline
  $\cC_{1,2,3,4}^{14}$ &  $x_3(a_3)x_1(a_1)x_2(a_2)x_4(a_4)x_{10}(a_2a_3a_4\mu)$  & $u_{14}^G$ \\
\end{longtabu}
\end{small}
\end{center}

\end{document}